\documentclass[letterpaper,12pt,reqno,twoside]{amsart}

\usepackage{amssymb}
\usepackage{amsmath}
\usepackage[colorlinks,citecolor=blue,urlcolor=blue]{hyperref}
\usepackage{cite}

\newcommand{\be}{\begin{eqnarray}}
\newcommand{\ee}{\end{eqnarray}}

\newcommand{\beq}{\begin{equation}}
\newcommand{\eeq}{\end{equation}}
\newcommand{\beqn}{\begin{equation*}}
\newcommand{\eeqn}{\end{equation*}}

\newcommand{\N}{\mathbb{N}}
\newcommand{\defas}{\mathrel{\raise.095ex\hbox{$:$}\mkern-4.2mu=}}
\newcommand{\defasr}{\mathrel{=\mkern-4.2mu\raise.095ex\hbox{$:$}}}

\newcommand{\norm}[1]{\lVert#1\rVert}

\newcommand{\average}[1]{\langle#1\rangle}

\newcommand{\ie}{\textit{i.e.}}

\DeclareMathAlphabet{\mathfat}{U}{bbold}{m}{n}          
\newcommand{\one}{\mathfat{1}}

\DeclareMathOperator{\dist}{\mathrm{dist}}

\newtheorem{thm}{Theorem}

\newtheorem{cor}[thm]{Corollary}
\newtheorem{lem}[thm]{Lemma}

\newtheorem{remark}[thm]{Remark}

\newtheorem{convention}[thm]{Convention}

\newcommand\cC{{\mathcal C}}

\newcommand\cE{{\mathcal E}}
\newcommand\cF{{\mathcal F}}
\newcommand\cG{{\mathcal G}}

\newcommand\cJ{{\mathcal J}}
\newcommand\cL{{\mathcal L}}
\newcommand\cM{{\mathcal M}}

\newcommand\cO{{\mathcal O}}
\newcommand\cP{{\mathcal P}}

\newcommand\cT{{\mathcal T}}
\newcommand\cU{{\mathcal U}}

\newcommand\cW{{\mathcal W}}

\newcommand\bA{{\mathbb A}}

\newcommand\bD{{\mathbb D}}
\newcommand\bE{{\mathbb E}}

\newcommand\bN{{\mathbb N}}

\newcommand\bR{{\mathbb R}}

\newcommand\fA{{\mathfrak A}}
\newcommand\fB{{\mathfrak B}}

\newcommand\fI{{\mathfrak I}}

\newcommand\fM{{\mathfrak M}}

\newcommand\fR{{\mathfrak R}}

\newcommand\fT{{\mathfrak T}}

\newcommand\fe{{\mathfrak e}}
\newcommand\ff{{\mathfrak f}}

\newcommand\bfh{{\mathbf h}}

\newcommand{\id}{{\mathrm{id}}}

\newcommand{\ve}{\varepsilon}

\newcommand{\slot}{{\text{\large .}}}

\begin{document}

\title{Non-stationary compositions of Anosov diffeomorphisms}

\author[Mikko Stenlund]{Mikko Stenlund}
\address[Mikko Stenlund]{
Courant Institute of Mathematical Sciences\\
New York, NY 10012, USA; Department of Mathematics and Statistics, P.O. Box 68, Fin-00014 University of Helsinki, Finland.}
\email{mikko@cims.nyu.edu}
\urladdr{http://www.math.helsinki.fi/mathphys/mikko.html}


\date{}

\begin{abstract}
Motivated by non-equilibrium phenomena in nature, we study dynamical systems whose time-evolution is determined by non-stationary compositions of chaotic maps. The constituent maps are topologically transitive Anosov diffeomorphisms on a 2-dimensional compact Riemannian manifold, which are allowed to change with time --- slowly, but in a rather arbitrary fashion.  In particular, such systems admit no invariant measure.
By constructing a coupling, we prove that any two sufficiently regular distributions of the initial state converge exponentially with time. Thus, a system of the kind loses memory of its statistical history rapidly.
\end{abstract}

\maketitle


\subsection*{Acknowledgements}
The author has received financial support from the Academy of Finland and the V\"ais\"al\"a Fund. He wishes to thank Lai-Sang Young for many discussions.


\section{Introduction} 

\subsection{Motivation}
Statistical properties of dynamical systems are traditionally studied in a stationary context. Let us elaborate briefly, discussing only discrete time for simplicity. Suppose $\cM$ is the set of all possible states of the system. Given the state $x_n\in\cM$ at some time $n\geq 0$, the state of the system at time $n+1$ is assumed to be either (i) $x_{n+1}=Tx_n$, where $T:\cM\to\cM$ is an \mbox{\textit{a priori}} specified map used at every time step or (ii) $x_{n+1}=T_{n+1} x_n$, where the maps $T_{i}:\cM\to\cM$ are drawn randomly and independently of each other and of $x_0,\dots,x_{i-1}$ from a set of maps $\fT$ according to a distribution $\eta$. Now, suppose the initial point $x_0$ has random distribution~$\mu$: $\operatorname{Prob}(x_0\in E)=\mu(E)$, for all measurable sets $E\subset \cM$.  By stationarity we mean that $\operatorname{Prob}(x_n\in E)=\mu(E)$, for \emph{all}~$n\geq 0$, for all measurable sets $E\subset \cM$. This condition translates to $\mu(T^{-1}E)=\mu(E)$ in case (i) and to $\int_\fT \mu(T^{-1}E)\, d\eta(T) = \mu(E)$ in case (ii). In either case the measure $\mu$ is called invariant. The reader may verify that these definitions result, indeed, in a \emph{strictly} stationary process in that all finite dimensional distributions are shift-invariant: $\operatorname{Prob}(x_{k_1+n}\in E_1,\dots,x_{k_m+n}\in E_m)=\operatorname{Prob}(x_{k_1}\in E_1,\dots,x_{k_m}\in E_m)$, for all choices of the indices and of the measurable sets. Let $f$ be a measurable function on $\cM$, which represents a quantity whose observed values $f(x_n)$ at different times one is interested in. Given an invariant measure one may, for example, study the statistical behavior of the sums $\sum_{i=0}^{n-1}f(x_i)$ of observations, making use of the fact that $(f(x_n))_{n\geq 0}$ is a stationary sequence of random variables.

A key ingredient in obtaining advanced statistical results on interesting systems is chaos, that is to say the dynamical complexity due to sensitive dependence of the trajectories $(x_n)_{n\geq 0}$ on the initial point $x_0$. In this paper we initiate a program to free ourselves from the standard constraint of stationarity, advocating the following view:
\begin{quote}
\begin{it}
Much of the statistical theory of stationary dynamical systems can be carried over to sufficiently chaotic non-stationary systems.  
\end{it}
\end{quote}
The deliberately imprecise statement above is proposed as a guideline and challenge instead of a theorem. We believe that a result obtained for a strongly chaotic stationary system quite generically has a non-stationary counterpart if the corresponding non-stationary system continues to be sufficiently chaotic.

The inspiration for undertaking the program stems from non-equilibrium processes in nature where it is often unfounded or simply false to assume that an observed system is driven by stationary forces. For example, it is conceivable that an ambient system governed in principle by measure preserving dynamics is, for all practical time scales, in a non-equilibrium state, so that the subsystem actually being observed is better modeled separately in terms of non-stationary dynamical rules.
The remark is by no means limited to situations of physical interest alone, but seems to lend itself rather universally to applied sciences. Second, from a purely theoretical point of view it appears very restrictive to focus only on stationary dynamical models.  

In order to advance the program in a meaningful way, we need a concrete model to work with. Deferring technical definitions till later, let the state $x_n\in\cM$ of the system at time $n$ be determined by the action of the composition $T_n\circ \dots\circ T_1$ on the initial state $x_0\in\cM$, where each map $T_i:\cM\to\cM$ describes the dynamical rules at time $i$. 
For us, the constituent maps $T_i$ are topologically transitive Anosov diffeomorphisms on a 2-dimensional compact Riemannian manifold $\cM$, which form a prime class of nontrivial chaotic maps. It is clear that some additional control is needed; for instance, an alternating sequence of an Anosov diffeomorphism $T$ and its inverse $T^{-1}$ would yield $T_n\circ \dots\circ T_1=\id$ for even values of $n$, which does not result in chaotic dynamics. To that end, the maps $T_i$ are here assumed to evolve slowly with time $i$, but otherwise they may do so in a rather arbitrary fashion. 
We point out that the maps $T_i$ need not be randomly picked, there is no assumption of stationarity, and for large~$n$ the map~$T_n$ may be far from the map~$T_1$. Even if all the maps $T_i$ preserved the same initial measure, so that the random variables $x_n$ were identically distributed, the process $(x_n)_{n\geq 0}$ would typically fail to be  stationary.

We call such compositions \emph{non-stationary} and think of them as descriptions of \emph{dynamical systems out of equilibrium}.

We prove in this paper that the system at issue loses memory of its initial state exponentially. More accurately, assume $x_0$ has either distribution $\mu^1$ or $\mu^2$ and call $\mu^1_n$ and $\mu^2_n$, respectively, the corresponding distributions of $x_n$. Our main result states that if $\mu^1$ and $\mu^2$ are sufficiently regular, then the difference $ \int f\,d \mu^1_n - \int f\,d \mu^2_n$ tends to zero at an exponential rate with increasing $n$, provided $f$ is a suitable test function. This type of weak convergence is natural due to the invertibility of the dynamics: the supports of the measures $\mu^1_n$ and $\mu^2_n$ will never overlap unless they did so initially. Instead, they tend to concentrate increasingly on unstable manifolds due to the contracting direction of the maps $T_i$ and then wind wildly around the phase space $\cM$ due to the expansion on unstable manifolds. Hence, one cannot hope to identify ever-increasing portions of $\mu^1_n$ and $\mu^2_n$ unless one first integrates against a test function that  possesses some regularity along stable manifolds. In spite of the convergence of the difference $\mu_n^1 - \mu_n^2$ for arbitrary initial measures $\mu^1$ and $\mu^2$, in general the limit measures $\lim_{n\to\infty}\mu_n^i$ do not exist individually even in the weak sense. It is more appropriate to think that all regular measures are attracted by a moving target in the space of measures. 

Finally, let us point out that in the real world, where observations take place on finite time scales, one is not interested in the excessively distant future. To underline this, the results here are finite-time results, in which the sequence $T_1,\dots,T_n$ is assumed to be known only up to some finite value of $n$. The lack of infinite future leads to certain technical problems to be discussed and dealt with below.

In \cite{OttStenlundYoung} analogous results were obtained for uniformly expanding and piecewise expanding maps. The situation of the present paper is markedly more complicated because our Anosov diffeomorphisms have a contracting direction. Some steps in this direction were taken in \cite{AyyerStenlund}, where mixing for certain arbitrarily ordered compositions of finitely many toral automorphisms was established. 
There are other studies which contain at least some elements that in spirit are not very far from our setting. In \cite{Bakhtin1,Bakhtin2} compositions of hyperbolic maps --- all close to each other --- were studied and limit theorems proved.   An abstract operator theoretic approach for obtaining limit theorems was described in \cite{ConzeRaugi}, with applications to piecewise expanding interval maps. Moreover, symbolic dynamics of non-stationary subshifts of finite type was considered in \cite{ArnouxFisher}. An extensive literature on random compositions of maps exists. It will not be reviewed here, as the present paper concerns quite a different type of questions. Nevertheless, some of the techniques developed below should be useful in the context of random maps as well.

\subsection{Structure of the paper}
In Section~\ref{subsec:Anosov_comp} we describe the precise setting of the paper. In particular, we explain what kind of compositions of maps we are interested in and discuss our standing assumptions. After that, the main result of the paper, Theorem~\ref{thm:weak_conv}, is formulated. Section~\ref{subsec:standard} introduces some basic concepts needed throughout the paper. The Introduction ends with Section~\ref{subsec:contributions}, which discusses what the author perceives as the most important contributions of the paper, including a technical version of Theorem~\ref{thm:weak_conv}.

In Section~\ref{sec:foliations} we define finite-time stable and unstable distributions and stable foliations needed to keep track of the dynamics with appropriate accuracy. We also prove quantitative results concerning the distortion effects of the dynamics. Subsequently, we are able to define in a meaningful way finite-time holonomy maps which satisfy useful bounds.

In Section~\ref{sec:coupling} we formulate the central result of the paper --- the Coupling Lemma. It is then used to prove Theorem~\ref{thm:equidistr}, which subsequently implies Theorem~\ref{thm:weak_conv}.
The Coupling Lemma itself is proved in Section~\ref{sec:coupling_proof}, which is the most technical part of the paper. 

To maintain the flow of the discussion, some key technical facts have been separated from the main text and presented in the appendices. They are cited in the text as needed. Appendix~\ref{app:Holder} is of special interest; there we prove the uniform H\"older regularity of the finite-time stable and unstable distributions introduced in Section~\ref{sec:foliations}.

\subsection{Compositions of Anosov diffeomorphisms}\label{subsec:Anosov_comp}
Fix $Q\in\N$. For each $1\leq q\leq Q$, let $\widetilde T_q:\cM\to\cM$ be a topologically transitive $\cC^2$ Anosov diffeomorphism on the 2-dimensional compact Riemannian manifold $\cM$ with metric $d$ embedded in an ambient space $\bR^M$ \footnote{Such a diffeomorphism is topologically conjugate to an automorphism of the torus.}. The Riemannian volume is denoted by $m$. The map $\widetilde T_q$ admits an invariant Sinai--Ruelle--Bowen (SRB) measure, $\mu_q$, which is mixing; see for example~\cite{BressaudLiverani}. In general, such a measure is not absolutely continuous with respect to the Riemannian volume. Let $\cU_q = \bD(\widetilde T_q,\ve_q)$ be disk neighborhoods of small radii $\ve_q>0$ in the $\cC^2$ topology. Now, pick a finite sequence $(T_n)$ of Anosov diffeomorphisms such that 
\beq\label{eq:nhood}
T_n\in\cU_q\qquad \forall\,n\in I_q = (n_{q-1},n_q], 
\eeq
where $0 = n_0 < n_1 <\dots \leq n_Q$. For technical reasons, also set $T_n = \widetilde T_Q$ for all $n > n_Q$. We assume that the intervals $I_q$ are long enough:
\beq\label{eq:time}
|I_q| = n_q-n_{q-1} \geq N_q,
\eeq
where the numbers $N_q$, $1\leq q\leq Q$, will be assumed suitably large. 
We will be interested in the statistical properties of the compositions 
\beq\label{eq:composition}
\cT_n=T_n\circ\dots\circ T_1\qquad
n\leq n_Q.
\eeq
The maps $\widetilde T_q$ serve as successive guiding points which the sequence $(T_n)$ follows in the space of Anosov diffeomorphisms, spending a sufficiently long time $N_q$ in each neighborhood $\cU_q$ before moving on to $\cU_{q+1}$. We also write $ \cT_{n,m} = T_n \circ\dots\circ T_m$ for $m\leq n$.

Each $\widetilde T_q$ admits a unique continuous invariant splitting of the tangent bundle: for each $x\in\cM$, $T_x\cM = E^u_{q,x}\oplus E^s_{q,x}$, where the 1-dimensional linear spaces $E^{u,s}_{q,x}$ depend continuously on the base point $x$, $D_x \widetilde T_q E^u_{q,x} = E^u_{q,\widetilde T_qx}$ and $D_x \widetilde T_q E^s_{q,x} = E^s_{q,\widetilde T_q x}$. In fact, in our 2-dimensional setting, the dependence on the base point is $\cC^{1+\alpha}$ for some $\alpha>0$, because the so-called bunching conditions \cite{KatokHasselblatt} are satisfied. The families $E^u_q=\{E^u_{q,x}\}$ and $E^s_q=\{E^s_{q,x}\}$ are called the unstable and stable distributions of $\widetilde T_q$, respectively, and their integral curves are called unstable and stable manifolds of $\widetilde T_q$, respectively. By continuity, the angle between $E^u_q$ and $E^s_q$ at each point is uniformly bounded away from zero.
The maps also have continuous families of unstable cones, $\{\cC^u_{q,x}\}$, and stable cones, $\{\cC^s_{q,x}\}$. These can be defined by setting
\beqn
\begin{split}
\cC^u_{q,x} &= \{v^u+v^s\,:\,v^u\in E_{q,x}^u,\, v^s\in E_{q,x}^s,\, \norm{v^s} \leq a_q \norm{v^u}\}, \\
\cC^s_{q,x} &= \{v^u+v^s\,:\,v^u\in E_{q,x}^u,\, v^s\in E_{q,x}^s,\, \norm{v^u} \leq a_q \norm{v^s}\},
\end{split}
\eeqn
for some constants $a_q>0$ such that
\begin{enumerate}
\item[(C1)] $D_x \widetilde T_q^n  \cC^u_{q,x} \subset \{0\}\cup \operatorname{int}  \cC^u_{q,\widetilde T_q^n x}$ and $D_x \widetilde T_q^{-n} \cC^s_{q,x} \subset \{0\}\cup  \operatorname{int}  \cC^s_{q,\widetilde T_q^{-n}x}$ if $n
\geq p_q$,
\item[(C2)] $\norm{D_x \widetilde T_q^n v} \geq \widetilde C_q \widetilde \Lambda_q^n \norm{v}$ if $v\in  \cC^u_{q,x} $ and $\norm{D_x \widetilde T_q^{-n} v} \geq \widetilde C_q \widetilde\Lambda_q^n \norm{v}$ if $v\in  \cC^s_{q,x} $,
\end{enumerate}
for constants $p_q\geq 1$, $0<\widetilde C_q<1$, and $\widetilde \Lambda_q >1$.

We make the following standing assumptions:
\begin{itemize}
\item[\bf(A0)] $p_q = 1$ in condition (C1) above.
\item[\bf(A1)] $D_x T \cC^u_{q,x} \subset \{0\}\cup  \operatorname{int}  \cC^u_{q,Tx}$ and $D_x T^{-1} \cC^s_{q,x} \subset \{0\}\cup  \operatorname{int}  \cC^s_{q,T^{-1}x}$ for all $T\in \cU_q$.
\item[\bf(A2)] There exist constants $0<C_q<1$ and $\Lambda_q>1$ such that, if each $T_i\in \cU_q$ for a fixed $q$, $\norm{D_x \cT_n v} \geq  C_q  \Lambda_q^n \norm{v}$ if $v\in  \cC^u_{q,x} $ and $\norm{D_x \cT_n^{-1} v} \geq C_q \Lambda_q^n \norm{v}$ if $v\in  \cC^s_{q,x} $.
\item[\bf(A3)] $D_x T \cC^u_{q,x} \subset \{0\}\cup  \operatorname{int}  \cC^u_{q+1, Tx}$ if $T\in \cU_{q+1}$ and $D_x T^{-1} \cC^s_{q+1,x} \subset \{0\}\cup  \operatorname{int}  \cC^s_{q,T^{-1}x}$ if $T\in \cU_q$.
\item[\bf(A4)] The numbers $a_q$ can be assumed small. 
\end{itemize}

\begin{convention}
From now on we will assume that $Q$ reference Anosov diffeomorphisms $\widetilde T_1,\dots,\widetilde T_Q$ have been fixed. When we say that a result does not depend on the choice of the sequence $(T_i)$, we mean that the result holds true uniformly for all finite sequences $(T_i)_{i=1}^{n_Q}$ of any length $n_Q$, provided \eqref{eq:nhood} and Assumptions (A) are satisfied and the numbers $N_q$ appearing in \eqref{eq:time} are large enough.
\end{convention}

Given $0<\gamma<1$, we say that a function $f:\cM\to \bR$ is a $\gamma$-H\"older continuous observable, if
\beqn
{|f|}_\gamma \equiv \sup_{x\neq y}\frac{|f(x)-f(y)|}{d(x,y)^\gamma} < \infty.
\eeqn

We are now in position to state our main theorem, which is reminiscent of weak convergence of measures in  probability theory. 
\begin{thm}[Weak convergence]\label{thm:weak_conv}
There exist constants $0<\eta<1$ and $C>0$, for which the following statements hold.
Let $d\mu^i = \rho^i dm$ ($i=1,2$) be two probability measures, absolutely continuous with respect to the Riemannian volume $m$, such that $\rho^i$ are strictly positive and $\eta$-H\"older continuous on $\cM$. If $f$ is continuous, then
\beqn
\left|\int_{\cM} f\circ \cT_n\, d\mu^1 - \int_{\cM} f\circ \cT_n\, d\mu^2\right| \leq A_f(n),  \qquad n\leq n_Q,
\eeqn
where $A_f(n)=o(1)$. Given $0<\gamma<1$, there exist constants $0<\theta_\gamma<1$ and $C_\gamma=C_\gamma(\rho^1,\rho^2)>0$ such that, if $f$ is $\gamma$-H\"older, then $A_f(n)=C_\gamma B_f\theta_\gamma^n$ with
$
B_f = C(\sup f-\inf f) +  {|f|}_\gamma$. In either case, the various constants do not depend on the choice of the sequence $(T_i)$, in particular its length $n_Q$, as long as the earlier assumptions hold and the numbers $N_q$ appearing in \eqref{eq:time} are large enough. Among the constants only $C_\gamma$ depends on the densities $\rho^i$, and in fact it only depends on the H\"older constants of $\ln\rho^i$. 
\end{thm}
In other words, if $f$ is continuous, the difference between the two integrals $\int_{\cM} f\circ \cT_n\, d\mu^i$ is eventually arbitrarily small, assuming there are sufficiently many maps in the finite sequence $(T_i)$ of $n_Q$ maps. The latter means that at least one of the intervals $I_q$ in \eqref{eq:time} is sufficiently long, and consequently $n_Q$ is large. What is more, the rate of convergence is exponential, if $f$ is H\"older continuous. By approximation, one can get an $o(1)$ estimate also for general continuous densities~$\rho^i$. 

Let us emphasize once more that despite such convergence or pairs of measures, it does not make any sense to speak of a limit measure, because the maps $T_n$ keep evolving with time --- possibly drifting very far from $T_1$. Furthermore, all observations in our theorems are restricted to times not exceeding (the arbitrarily large but finite) $n_Q$.

Theorem~\ref{thm:weak_conv} remains true for much more general, SRB-like, initial measures. It is enough that each measure $\mu^i$ can be disintegrated relative to a measurable partition $\cP^i$ such that the partition elements $W\in\cP^i$ are smooth unstable curves with respect to the cones $\{\cC_{1,x}^u\}$ with uniformly bounded curvatures and the conditional measures $\mu^i| W$ have regular densities. See below for details. 

In our formulation of the theorem, the convergence rate $\theta_\gamma$ is constant. The latter depends on the reference diffeomorphisms $\widetilde T_q$, $1\leq q\leq Q$. A sharper, variable, convergence rate that depends also on the time interval $I_q$ that $n$ belongs to, can be deduced from the proof.

We finish the section by discussing Assumptions (A) and how they could be relaxed.

Assumption (A0) is one of convenience; we could as well assume that $T_{p_q}\circ\dots\circ T_1$ is sufficiently close to $\widetilde T^{p_q}$, but have opted for a streamlined presentation. 
Assumptions (A1) and (A2) state that compositions of maps belonging to $\cU_q$ have similar hyperbolicity properties as powers of $\widetilde T_q$. The following lemma is proved after a few paragraphs: 
\begin{lem}\label{lem:ass}
Assumptions (A1) and (A2) are satisfied if $\ve_q$ is sufficiently small.
\end{lem}

Assumption (A3) guarantees that hyperbolicity prevails when a transition from $\cU_q$ to $\cU_{q+1}$ occurs. The first part of (A3) could be relaxed by replacing the map $T\in \cU_{q+1}$ by sufficiently long compositions $\cT_n=T_n\circ\cdots\circ T_1$ of maps with each $T_i\in \cU_{q+1}$: given a sufficiently large $r_q>0$,  $D_x \cT_n \cC^u_{q,x} \subset \{0\}\cup  \operatorname{int}  \cC^u_{q+1, \cT_nx}$ if $n\geq r_q$ and $T_i\in \cU_{q+1}$ for $1\leq i \leq n$. The assumption is then satisfied, for example, if $\cU_q\cap\cU_{q+1}\neq\varnothing$ and if $\ve_q$ is small, for all $q$. However, $\cU_q$ and $\cU_{q+1}$ need not overlap or even be close to each other, as most vectors in the tangent space $T_x\cM$ get eventually mapped by $\cT_n=T_n\circ\dots\circ T_1$ into $\cC^u_{q+1,\cT_n x}$ if each $T_i\in \cU_{q+1}$ and if $\ve_{q+1}$ is small. Similar remarks hold for the second part of (A3). This way, the sequence $(T_n)$ used to build up the compositions \eqref{eq:composition} might, without affecting our analysis, involve occasional long jumps from one neighborhood $\cU_q$ to the next, as long as the number of steps $|I_q|$ spent in each neighborhood $\cU_q$, see \eqref{eq:time}, is sufficiently large.

Assumption (A4) means that the cones can be assumed narrow. This is not restrictive for our purposes either, as it follows from (C2) that arbitrarily narrow cones can be treated by considering sufficiently long compositions of maps in a given $\cU_q$ with a sufficiently small $\ve_q$.

Recapitulating, it would be adequate to assume that the properties above hold eventually, for sufficiently long compositions of maps, and in this case the assumptions are very natural and easily fulfilled. For technical convenience and notational ease, we assume from now on that all the nice properties hold immediately, after the application of just one map.

\begin{proof}[Proof of Lemma~\ref{lem:ass}]
We will prove the claims for unstable cones, going forward in time. Similar arguments work for the stable cones, by reversing time.

(A1): Suppose $T \in \cU_q$. By the chain rule $D_x T \cC_{q,x}^u = D_{\widetilde T_q x} (T \widetilde T_q^{-1}) D_x \widetilde T_q \, \cC_{q,x}^u$, where $D_x \widetilde T_q \cC_{q,x}^u\subset \{0\}\cup \operatorname{int} \cC_{q,\widetilde T_q x}^u$. By the continuity of the cones with respect to the base point and the fact that $D_{\widetilde T_q x} (T\widetilde T_q^{-1}) = \one + \cO(\ve_q)$ \footnote{Here it is understood that $\cM$ is embedded in the ambient space $\bR^M$ and that $D_{\widetilde T_q x} (T\widetilde T_q^{-1})$ acts between the linear subspaces $T_{\widetilde T_q x}\cM$ and $T_{T x}\cM$ of $\bR^M$.}, we have $D_x T \cC_{q,x}^u\subset \{0\}\cup \operatorname{int} \cC_{q,Tx}^u$, provided  $\ve_q$ is sufficiently small. Compactness guarantees that $\ve_q$ can be chosen independently of $x$.

(A2): For each $i$ and $x$, $D_xT_i = D_x\widetilde T_q + \cE_{i,x}$, where $\sup_{i,x}\norm{\cE_{i,x}} = \cO(\ve_q)$. We can bound $\norm{D_x \cT_N - D_x \widetilde T_q^N}\leq C(N) \ve_q$. If $\ve_q$ is sufficiently small, we have $\norm{D_x \cT_N v} \geq \norm{D_x \widetilde T_q^N v} - C(N)\ve_q \norm{v}\geq \frac{1}{2}\widetilde C_q\widetilde \Lambda_q^N\norm{v}$ for  $v\in \cC^u_{q,x}$. Now assume $N=N(q)$ is so large that $\frac{1}{2}\widetilde C_q\widetilde \Lambda_q^N>1$. Here $N$ depends neither on $x$, on $v$, nor on the choice of the maps $T_i \in \cU_q$. Let us set $\Lambda_q = (\frac{1}{2}\widetilde C_q\widetilde \Lambda_q^N)^{1/N}$. The uniform estimate $\norm{D_x \cT_n v} \geq  c_q \norm{v}$ holds with some $c_q=c_q(N)$ for $1\leq n < N$. Now, assume $n = kN+l$, $0\leq l <N$. Then $\norm{D_x \cT_n  v} \geq  c_q \norm{D_x\cT_{kN}  v} \geq c_q \Lambda_q^{kN} \norm{v}$. Thus, we can take $C_q = c_q/\Lambda_q^N$. 
\end{proof}

\subsection{Unstable curves with smooth measures}\label{subsec:standard}
We call a smooth curve $W\subset \cM$ unstable with respect to $\{\cC^u_{q,x}\}$ if its tangent space at each point $x\in W$ is contained in the unstable cone $\cC^u_{q,x}$, \ie, $T_xW\subset \cC^u_{q,x}$. Stable curves are defined similarly. Let $W(x,y)\subset W$ denote the subcurve of $W$ whose end points are $x,y\in W$. The length $|W|$ of a curve $W$ is given by
\beqn
|W| = \int_W dm_W,
\eeqn 
where $m_W$ stands for the measure $m_W$ on $W$ induced by the Riemannian metric. Also, let $\kappa(W)$ stand for the maximum curvature of $W$: if $u(x)$ is a unit tangent vector of $W$ at $x$ depending smoothly on $x$, then $\kappa(W)=\sup \|u\cdot\nabla u\|$.

It is convenient to consider curves of bounded length and curvature only. Hence, we introduce two length caps, $L$ and $\ell<L$, and a curvature cap $K$, and say that a smooth curve $W$ is \emph{standard}, if $\ell\leq |W|\leq L$ and if $\kappa(\cT_nW)\leq K$ for all $n\geq 0$. If $|W|>L$, we can always ``standardize'' it by cutting it into shorter subcurves. If an unstable curve is of length less than $\ell$, it will eventually grow under the application of the sequence $(T_i)$, such that $|\cT_n W|\geq \ell$ for sufficiently large $n$. The dynamics also flattens unstable curves, such that $\kappa(\cT_n W)\leq K$ for all sufficiently large $n$, even if $\kappa(W)>K$. We will confirm these last two facts in the following. Finally, there are no discontinuities which would introduce more short curves under the dynamics by cutting longer ones. Taking these considerations into account it is quite natural to commit to the mild constraint that all curves are standard curves to begin with. This will help keep the somewhat technical discussion as clear as possible.

A \emph{standard pair} $(W,\nu)$ (w.r.t.\ $\{\cC^u_{q,x}\}$) consist of a standard unstable curve (w.r.t.\ $\{\cC^u_{q,x}\}$), $W$, and a probability measure, $\nu$, on $W$. The measure $\nu$ is assumed absolutely continuous with respect to $m_W$ on $W$ with a density, $\rho$, that is regular in the following sense: for some global constants $C_\mathrm{r}>0$ and $\eta_\mathrm{r}\in(0,1]$ to be fixed later \footnote{$C_\mathrm{r}$ has to satisfy the condition in Lemma \ref{lem:standard_image} and $\eta_\mathrm{r}$ is determined in Lemma \ref{lem:tau_reg}. Both depend on the reference sequence $\widetilde T_1,\dots, \widetilde T_Q$, but not on the choice of $(T_i)$.},
\beq\label{eq:regular_density}
|\ln \rho(x) - \ln \rho(y)| \leq C_\mathrm{r} |W(x,y)|^{\eta_\mathrm{r}}
\eeq
for all $x,y\in W$.
In particular,
\beqn
\frac{\sup \rho}{\inf\rho} \leq e^{C_\mathrm{r}|W|^{\eta_\mathrm{r}}},
\eeqn
which implies $\inf\rho\geq \frac{1}{|W|}e^{-C_\mathrm{r}{L^{\eta_\mathrm{r}}}}\geq \frac{1}{L}e^{-C_\mathrm{r}{L^{\eta_\mathrm{r}}}}>0$, since $\int_W\rho\,dm_W =1$. Moreover, $\inf\rho\leq \frac{1}{|W|}$. If $W' \subset W$, we obtain by using the previous facts that
\beqn
e^{-C_\mathrm{r}L^{\eta_\mathrm{r}}} \leq \frac{|W|}{|W'|}\nu(W') \leq e^{C_\mathrm{r}L^{\eta_\mathrm{r}}} .
\eeqn
Hence, if $D=e^{2C_\mathrm{r}L^{\eta_\mathrm{r}}}$ and $W',W''\subset W$,
\beq\label{eq:density_comp}
D^{-1} \frac{\nu(W'')}{|W''|} \leq \frac{\nu(W')}{|W'|} \leq D\, \frac{\nu(W'')}{|W''|}.
\eeq

Formally, a standard family is a family  $\cG=\{(W_\alpha,\nu_\alpha)\}_{\alpha\in\fA}$ of standard pairs together with a probability factor measure $\lambda_\cG$ on the (possibly uncountable) index set $\fA$ and a probability measure $\mu_\cG$ satisfying
\beqn
\mu_\cG(B) = \int_{\fA} \nu_\alpha(B\cap W_\alpha) \,d\lambda_\cG(\alpha)
\eeqn
for each Borel measurable set $B\subset \cM$. The measure  $\mu_\cG$ is supported on $\cup_\alpha W_\alpha$ and
\beqn
\bE_{\cG}(f) = \int_{\cM}f \,d\mu_\cG = \int_\fA\int_{W_\alpha} f(x) \,d\nu_\alpha(x) d\lambda_\cG(\alpha)
\eeqn
for each Borel measurable function $f$ on $\cM$. In Theorem~\ref{thm:equidistr} we assume that a standard family is associated to a measurable partition.

A standard family can, for example, consist of just one standard pair $\{(W,\nu)\}$ and the Dirac point mass factor measure $\delta_{W}$. Another natural example of a standard family $\{(W_\alpha,\nu_\alpha)\}_{\alpha\in\fA}$ is obtained by considering an Anosov diffeomorphism and taking as $\{W_\alpha\}_{\alpha\in\fA}$ a measurable partition consisting of  unstable manifolds of bounded length and letting the Riemannian volume induce the factor measure and the conditional measures $\nu_\alpha$.

\subsection{Main contributions}\label{subsec:contributions}
A technical version of our main result is the following theorem. It states that for reasonable initial distributions $\mu_\cG$ and $\mu_\cE$, the images $\cT_n \mu_\cG$ and $\cT_n \mu_\cE$ converge exponentially in a weak sense.
\begin{thm}\label{thm:equidistr}
There exist constants $C>0$ and $0<\vartheta<1$, and $0<\lambda<1$, such that the following holds. For any standard families $\cG$ and $\cE$, any $\gamma>0$, and any $\gamma$-H\"older observable~$f$~\footnote{Given a sequence $(T_i)$, it is in fact enough to assume that $f$ is H\"older continuous along the finite-time stable leaves associated to that particular sequence; see Section~\ref{sec:foliations}.},
\beqn
\left|\int_{\cM} f\circ \cT_n\,d\mu_{\cG} - \int_{\cM} f\circ \cT_n\,d\mu_{\cE}\right| \leq B_f\theta_\gamma^n,  \qquad n\leq n_Q,
\eeqn
where, 
\beqn
B_f = C(\sup f-\inf f) +  {|f|}_\gamma \quad\text{and}\quad \theta_\gamma = \max(\vartheta, \lambda^\gamma)^{1/2}.
\eeqn
The various constants do not depend on the choice of the sequence $(T_i)$, in particular its length~$n_Q$, as long as the earlier assumptions hold and the numbers $N_q$ appearing in \eqref{eq:time} are large enough.
\end{thm}

As a consequence of Theorem~\ref{thm:equidistr}, we prove the earlier Theorem~\ref{thm:weak_conv}, which is stated in terms of less technical notions and is closer in spirit to the weak convergence of probability theory.

The proofs of Theorems~\ref{thm:equidistr} and~\ref{thm:weak_conv} rely on a coupling method that has its roots in probability theory. It was carried over to the study of dynamical systems by Lai-Sang Young \cite{Young,MasmoudiYoung} who used it to prove exponential decay of correlations for Sinai Billiards and uniqueness of invariant measures for randomly perturbed dissipative parabolic PDEs. Bressaud and Liverani \cite{BressaudLiverani} also used coupling to give explicit estimates on the decay of correlations for Anosov diffeomorphisms. The present paper takes advantage of a version of Young's coupling method introduced by Dmitry Dolgopyat and Nikolai Chernov \cite{ChernovDolgopyatBBM,Chernov-BilliardsCoupling}.

A considerable amount of work is devoted to obtaining \emph{uniform bounds}, which is more involved than in the case of iterating a single map. A central issue is that the finite sequences $(T_1,\dots,T_{n_Q})$ of maps that we consider do not possess stable and unstable manifolds, because defining such objects requires an infinite future and an infinite past, respectively. Thus, we have to resort to artificial, finite-time, foliations that describe the dynamics sufficiently faithfully but are by no means unique. Moreover, in the single map case the regularity properties of the foliations of the manifold into stable and unstable manifolds play an important role. Our construction should therefore also yield regular foliations. In addition, the amount of regularity must not depend on the choice of the sequence $(T_1,\dots,T_{n_Q})$ (as long as $Q\geq 1$ and the maps $\widetilde T_q$, $1\leq q\leq Q$, have been fixed and the earlier assumptions are satisfied), since the goal is to prove the uniform convergence result in Theorem~\ref{thm:equidistr}.  

At the heart of Dolgopyat's and Chernov's method lies the Coupling Lemma (corresponding to Lemma~\ref{lem:coupling}). In its proof, one constructs a special reference set called the magnet. By mixing, any standard pair will ultimately cross the magnet as if it was attracted by the latter. Once two standard pairs cross the magnet, parts of them can be coupled to each other using the stable foliation. In this paper, we generalize the idea by considering time-dependent magnets and time-dependent, finite-time, foliations for the coupling construction.


\section{Distortions and holonomy maps}\label{sec:foliations}
\subsection{Stable foliations $\cW^n$}
As pointed out above, there is no well-defined sequence of stable foliations associated to the finite sequence $(T_1,\dots,T_{n_Q})$ of maps. A way around this is to try to augment the sequence with a fake future consisting of infinitely many maps --- in our case $T_n = \widetilde T_Q$ for $n>n_Q$ --- and to consider the uniquely defined stable foliations of the resulting infinite sequence of maps. This sequence of stable foliations naturally depends on the chosen future and it is not \textit{a priori} clear whether they have very much to do with the finite-time dynamics ($1\leq n\leq n_Q$) which is the only thing we are interested in. 

For a sequence $(T_i)$ satisfying the earlier assumptions, we can define a sequence of stable distributions on the manifold $\cM$, by pulling back the stable distribution $E_{Q,x}^s$ of $\widetilde T_Q$. More precisely, let us first define $E_x^{n} = E_{Q,x}^s$ for $n\geq n_Q+1$ and then
\beqn
E_x^n = D_{T_{n+1} x} T_{n+1}^{-1}E_{T_{n+1} x}^{n+1},  \qquad 0\leq n\leq n_Q.
\eeqn
With this definition,
\beqn
D_x \cT_{n,m} E_x^{m-1} = E_{\cT_{n,m} x}^n,\qquad n\geq m\geq 1.
\eeqn
Assumptions (A1) and (A3) guarantee that $E_x^n \subset \{0\}\cup \operatorname{int} \cC^{s}_{q,x}$ for $n+1\in I_q$  and $1\leq q\leq Q$. By Assumption (A4), the angle between $E_x^n$ and $E_{q,x}^s$ can be assumed uniformly small,  for $n\in I_q$ and $1\leq q\leq Q$.

The distributions $E^n$ above are the tangent distributions to the stable foliations $\cW^n$ of the sequences $(T_i)_{i> n}$. If $\cW^s_{Q,x}$ is the stable leaf of $\widetilde T_Q$ at $x$, then $\cW^{n}_x = \cW^s_{Q,x}$ for $n\geq n_Q+1$ and 
\beqn
\cW^n_x = T_{n+1}^{-1} \cW^{n+1}_{T_{n+1} x},  \qquad 0\leq n\leq n_Q.
\eeqn
Notice that $y\in \cW^n_x$ if and only if $\lim_{N\to\infty}d(\cT_{N,n+1}x,\cT_{N,n+1}y)=0$. 

For technical reasons, we also define $F_x^n = E_{1,x}^u$ for $n\leq 0$ and then
\beq\label{eq:fake_unstable}
F_x^n = D_{T_n^{-1} x} T_n F^{n-1}_{T_n^{-1} x},  \qquad 1\leq n\leq n_Q.
\eeq
Assumptions (A1) and (A3) guarantee that $F_x^n \subset \{0\}\cup \operatorname{int} \cC^{u}_{q,x}$ for $n\in I_q$ and $1\leq q\leq Q$. By Assumption (A4), the angle between $F_x^n$ and $E_{q,x}^u$ can be assumed uniformly small,  for $n\in I_q$ and $1\leq q\leq Q$. The distributions $F^n$ are in fact the unstable distributions of the sequence $(T_i)$ augmented with the past $T_i=\widetilde T_1$ for $i\leq 0$. They serve as \emph{H\"older continuous} reference distributions that allow us to accurately compare different unstable vectors.

In Appendix~\ref{app:Holder} we show that the distributions $F^n$ and $E^n$ for all $n$ are uniformly H\"older continuous.


\subsection{Distortion}
It is necessary to control the distortion and growth of curves under maps $T$. Given a curve $W$ we denote by $\cJ_WT$ the Jacobian of the restriction of $T$ to $W$. If $v$ is any nonzero tangent vector of $W$ at $x$, then
\beqn
\cJ_WT(x) = \frac{\norm{D_x T v}}{\norm v} .
\eeqn
\begin{lem}[Growth of unstable curves]\label{lem:growth_loc}
Fix $1\leq q\leq Q$ and let $T_i\in \cU_q$ for each $i$. If $W$ is an unstable curve with respect to $\{\cC^u_{q,x}\}$, setting $\bar\Lambda_q = \sup_x \norm{D_x\widetilde T_q}+\ve_q$,
\beq\label{eq:growth_bound}
C_q\Lambda_q^n|W| \leq |\cT_n W| \leq
\bar\Lambda_q^n|W|.
\eeq
If $W,\cT_1W,\dots, \cT_nW$ are stable curves with respect to $\{\cC^s_{q,x}\}$, then 
\beq\label{eq:contraction_bound}
 |\cT_n W| \leq \frac{|W|}{C_q\Lambda_q^n}.
\eeq
\end{lem}
\begin{proof}
Since $|\cT_n W| = \int_{\cT_n W}\,dm_{\cT_n W} = \int_W\cJ_W\cT_n\, dm_W = \int_W\norm{D_x\cT_n v_x}\, dm_W(x) $, where $v_x$ is a unit vector tangent to $W$ at $x$, it suffices to  observe that $C_q\Lambda_q^n\leq \norm{D_x\cT_n v_x}\leq \bar\Lambda_q^n$ in the ``unstable case'' and $\norm{D_x\cT_n v_x}\leq \frac{1}{C_q\Lambda_q^n}$ in the ``stable case".
\end{proof}

\begin{lem}[Curvature of unstable curves]\label{lem:curvature}
Fix $1\leq q\leq Q$ and let $T_i\in \cU_q$ for each $i$. There exist $K_1$ and, for any $K'>0$, $K_2(K')$ and $n_\kappa(K')$, such that
\beqn
\kappa(\cT_n W) \leq 
\begin{cases}
K_1, & n\geq n_\kappa, \\
K_2, & n\geq 0,
\end{cases}
\eeqn
holds if $W$ is an unstable curve with respect to $\{\cC^u_{q,x}\}$ and $\kappa(W)\leq K'$. Notice that  $K_1\leq K_2$ is independent of $K'$.
\end{lem}
\begin{remark}
We can now fix some $K'$ and set $K=K_2(K')$ in the definition of standard pairs. In particular, this means that any unstable curve $W$ with length between $\ell$ and $L$ and curvature $\kappa(W)\leq K'$ is a standard curve.
\end{remark}
\begin{proof}[Proof of Lemma~\ref{lem:curvature}]
Let $W$ be an unstable curve and $\gamma$ its parametrization by arc length, such that $u(x)=\dot\gamma(t)\in\cC^u_{q,x}$ with $x=\gamma(t)$. Note $\|u\|=1$. The curvature of $W$ at $x$ is the length of
\beqn
\ddot \gamma(t)= u(x)\cdot\nabla u(x).
\eeqn
Setting $V=D\cT_n u$, $v(y)=V(x) / \|V(x)\|$ is the unit tangent of $\cT_nW$ at $y=\cT_nx$. The curvature of $\cT_nW$ at $y$ is thus obtained from
\beqn
\begin{split}
v(y)\cdot \nabla v(y) &= Dv(y)v(y) = \|V(x)\|^{-1} Dv(y)D\cT_n(x)u(x) =  \|V(x)\|^{-1} D_x(v(y))u(x). 
\end{split}
\eeqn  
Here the chain rule $D_x(v(y)) = Dv(y)D\cT_n(x)$ was used. Now
\beqn
\begin{split}
 D_x(v(y)) u(x) &= D(\|V(x)\|^{-1}V(x)) u(x) = \|V(x)\|^{-1}DV(x) u(x) + V(x) D(\|V(x)\|^{-1}) u(x)
 \\
&= \|V(x)\|^{-1}DV(x) u(x) + V(x) \left( -\|V(x)\|^{-3}V(x)\cdot DV(x) \right) u(x)
 \\
&= \|V(x)\|^{-1}DV(x) u(x) - \|V(x)\|^{-1}v(y) \left(v(y)\cdot DV(x) u(x)\right),
 \end{split}
\eeqn
such that
\beqn
v(y)\cdot \nabla v(y) =  \|V(x)\|^{-2}\left( DV(x) u(x) - v(y) \left(v(y)\cdot DV(x) u(x)\right) \right),
\eeqn
or compactly
\beq\label{eq:curv_image}
v\cdot \nabla v =  \|V\|^{-2}\left( DV u - v \left(v\cdot DVu\right) \right).
\eeq
Notice that $DV u - v \left(v\cdot DVu\right)$ is the component of $DVu$ orthogonal to $v$ and hence $\| DV u - v \left(v\cdot DVu\right)  \| \leq \|DVu\|$. Furthermore, as $Du\,u = u\cdot\nabla u$, which we recognize to be the curvature of $W$ at $x$, we have
\beq\label{eq:DVu}
DVu =  D^2\cT_n(u,u)+D\cT_n(u\cdot\nabla u).
\eeq
Using Lemma~\ref{lem:exp_comp} and $\|D\cT_nu\|\geq C_q\Lambda_q^n\|u\|$, we see from \eqref{eq:curv_image} and \eqref{eq:DVu} that
\beq\label{eq:curvature_rec}
\begin{split}
\| v\cdot \nabla v \| & \leq  \|V\|^{-2} \| DVu \| \leq \frac{\|D^2\cT_n(u,u)\|}{\|D\cT_nu\|^2}+ \frac{\|D\cT_n(u\cdot\nabla u)\|}{\|D\cT_nu\|^2} 
\\
& \leq (C_q\Lambda_q^n)^{-2} {\norm{D^2\cT_n}}_\infty + (C_q\Lambda_q^n)^{-1} C_\#\|u\cdot\nabla u\|.
\end{split}
\eeq
Fix an $N$ such that $(C_q\Lambda_q^N)^{-1} C_\#<1$. Iterating \eqref{eq:curvature_rec},
\beqn
\kappa(\cT_{kN+l} W) \leq \frac{(C_q\Lambda_q^N)^{-2}{\sup_{(T_i)}\norm{D^2\cT_N}}_\infty}{1-(C_q\Lambda_q^N)^{-1} C_\#}  +((C_q\Lambda_q^N)^{-1} C_\#)^k  \kappa(\cT_l W).
\eeqn
A uniform bound $\max_{0\leq l<N}\kappa(\cT_l W)\leq a+b\cdot\kappa(W)$ is also obtained, so we are done. 
\end{proof}

If $W$ carries a measure $\nu$ with density $\rho$, then $\cT_n W$ carries the measure $\cT_n \nu$ whose density, which we denote $\cT_n\rho$, is
\beqn
(\cT_n\rho)(x) = \frac{\rho(\cT^{-1}_n x)}{\cJ_{W}\cT_n (\cT^{-1}_n x)} = \cJ_{\cT_nW} \cT^{-1}_n(x)\cdot\rho(\cT^{-1}_n x).
\eeqn
For controlling the regularity of such densities, we have the following result.
\begin{lem}[Distortion bound]\label{lem:distortion_loc}
Fix $1\leq q\leq Q$ and let $T_i\in \cU_q$ for each $i$. If $\cT_n^{-1}W$ is a standard unstable curve with respect to $\{\cC^u_{q,x}\}$ for all $0\leq n\leq N$ and if $x,y\in W$, then 
\beqn
\left |\ln\frac{\cJ_W \cT_n^{-1}(x)}{\cJ_W \cT_n^{-1}(y)} \right| \leq C_{\mathrm{d},q} \! \left| W(x,y)\right|, \qquad n\leq N.
\eeqn
Here $C_{\mathrm{d},q}>0$ is a constant that do not depend on $W$ or the choice of $(T_i)$.
\end{lem}
\begin{proof}
We first prove that the distortion factor of any map $T\in\cU_q$ is close to that of $\widetilde T_q$. To this end, let $W$ be an unstable curve with respect to $\{\cC^u_{q,x}\}$, $x\in W$, and $v$ a unit vector tangent to $W$ at $x$. Then
\beqn
\begin{split}
\left| \cJ_W T^{-1}(x) - \cJ_W \widetilde T_q^{-1}(x) \right| &= \left| \norm{D_x T^{-1} v} - \norm{D_x \widetilde T_q^{-1} v}\right| \leq  \norm{(D_x T^{-1}  - D_x \widetilde T_q^{-1})v}
\\
&= \norm{D_x T^{-1} (D_{\widetilde T_q^{-1}x}\widetilde T_q-D_{T^{-1}x}T)D_x \widetilde T_q^{-1}v}
\\
&\leq C\norm{D_{\widetilde T_q^{-1}x}\widetilde T_q-D_{T^{-1}x}T}\cJ_W \widetilde T_q^{-1}(x) 
\leq C\ve_q\, \cJ_W \widetilde T_q^{-1}(x).
\end{split}
\eeqn
Next, let $\gamma$ parametrize $W(x,y)$ according to arc length. Because
\beqn
\begin{split}
\left| \frac{d}{dt}\cJ_W T^{-1} (\gamma(t))\right| &= \left| \frac{d}{dt}\! \left\| D_{\gamma(t)}T^{-1} \dot\gamma(t) \right\| \right| = \left| \frac{D_{\gamma(t)}T^{-1} \dot\gamma(t)}{\left\|D_{\gamma(t)}T^{-1} \dot\gamma(t)\right\|}\cdot \frac{d}{dt}\! \left( D_{\gamma(t)}T^{-1} \dot\gamma(t) \right) \right|
\\
& \leq \left\| \frac{d}{dt}\! \left( D_{\gamma(t)}T^{-1} \dot\gamma(t) \right) \right\| = \left\| D^2_{\gamma(t)}T^{-1} (\dot\gamma(t),\dot\gamma(t)) + D_{\gamma(t)}T^{-1} \ddot\gamma(t) \right\|
\\
& \leq \sup_x \! \left\| D_x^2 T^{-1} \right\| + \sup_x \! \left\| D_x T^{-1} \right\|\! \|\ddot\gamma(t)\|,
\end{split}
\eeqn
and because the curvature $\|\ddot\gamma\|\leq K$ for all standard unstable curves,
\beqn
\begin{split}
\left| \ln\cJ_W T^{-1}(x) - \ln\cJ_W T^{-1}(y) \right| &= \left| \int_0^{|W(x,y)|} \frac{d}{dt}\! \left(\ln  \cJ_W T^{-1} (\gamma(t)) \right) dt \right|
\\
& =  \left| \int_0^{|W(x,y)|}\! \frac{\frac{d}{dt}\cJ_W T^{-1} (\gamma(t))}{\cJ_W T^{-1} (\gamma(t))}\, dt \right|  \leq \widetilde C_{\mathrm{d},q} |W(x,y)| ,
\end{split}
\eeqn
where $\widetilde C_{\mathrm{d},q}$ is independent of the choice of $T$. The desired estimate follows. Indeed, writing $x^{-j} = (\cT_{n,n-j+1})^{-1}x$, $y^{-j} = (\cT_{n,n-j+1})^{-1}y$, and $W^{-j} = (\cT_{n,n-j+1})^{-1}W$ (with $\cT_{n,n+1}=\id$),
\beqn
\begin{split}
\left |\ln\frac{\cJ_W \cT_n^{-1}(x)}{\cJ_W \cT_n^{-1}(y)} \right| &\leq \sum_{j=0}^{n-1} \left |\ln\frac{\cJ_{W^{-j}} T_{n-j}^{-1}(x^{-j})}{\cJ_{W^{-j}} T_{n-j}^{-1}(y^{-j})} \right|  \leq \sum_{j=0}^{n-1} \widetilde C_{\mathrm{d},q} |W^{-j}(x^{-j},y^{-j})| .
\end{split}
\eeqn
Moreover, by \eqref{eq:growth_bound},
$ |W^{-j}(x^{-j},y^{-j})| \leq C_q^{-1}\Lambda_q^{-j} |W(x,y)|$.
\end{proof}

\subsection{Image of a standard family}\label{sec:standard_image}

\begin{lem}\label{lem:standard_image}
Fix $1\leq q\leq Q$ and let $T_i\in \cU_q$ for each $i$. Let $\cG=(W,\nu)$ be a standard pair with respect to $\{\cC^u_{q,x}\}$ and assume that $C_\mathrm{r}$ satisfies
\beqn
2C_\mathrm{d,q}L^{1-{\eta_\mathrm{r}}} \leq C_\mathrm{r}.
\eeqn
For $n\geq \ln\frac{2}{C_q}/\ln\Lambda_q$, denote by $W_i$ the (finitely many) standard pieces of the image $\cT_n W$ after it has been standardized by cutting into shorter pieces and split the image measure $\cT_n\nu$ into the sum $\sum_i c_i\nu_i$, where $\nu_i$ is a probability measure on $W_i$ and $\sum_i c_i=1$. 
Then each $(W_i,\nu_i)$ is a standard pair w.r.t.\ $\{\cC^u_{q,x}\}$.
\end{lem}
\begin{proof} 
We only need to check that the density, $\rho_i$, of $\nu_i$ is regular. For $x\in W_i$, $\rho_i(x)=\cJ_{W_i}\cT_n^{-1}(x)\cdot \rho(\cT_n^{-1}x)/c_i$. Thus, for any pair $x,y\in W_i$,
\beqn
\begin{split}
|\ln\rho_i(x)-\ln\rho_i(y)| & \leq |\ln\rho(\cT_n^{-1}x)-\ln\rho(\cT_n^{-1}y)| + \left| \ln \frac{\cJ_{W_i}\cT_n^{-1}(x)}{\cJ_{W_i}\cT_n^{-1}(y)} \right| \\
& \leq C_\mathrm{r} |W(\cT_n^{-1}x,\cT_n^{-1}y)|^{\eta_\mathrm{r}} + C_\mathrm{d,q}|W_i(x,y)|\\
& \leq (C_\mathrm{r}C_q^{-{\eta_\mathrm{r}}}\Lambda_q^{-{\eta_\mathrm{r}}\cdot n} + C_\mathrm{d,q}L^{1-{\eta_\mathrm{r}}}) |W_i(x,y)|^{\eta_\mathrm{r}} \leq C_\mathrm{r} |W_i(x,y)|^{\eta_\mathrm{r}}.
\end{split}
\eeqn
We used Lemma~\ref{lem:distortion_loc} and also $W_i(x,y) = \cT_n( W(\cT_n^{-1} x, \cT_n^{-1} y) )$ together with Lemma~\ref{lem:growth_loc}.
\end{proof}
Thus, $\cG_n=\{(W_i,\nu_i)\}$ is a standard family equipped with the factor measure $\lambda_{\cG_n}(i)=c_i$. More generally, if $\cG$ is a standard family,  $\cG_n$ obtained by processing each standard pair in a similar fashion is a standard family.

\subsection{Holonomy maps}
A holonomy map is a device needed in the coupling construction for coupling some of the probability masses on different points. Let $W_1$ and $W_2$ be two unstable curves (w.r.t.\ $\{\cC^u_{1,x}\}$) connected by the stable foliation $\cW^0$. In other words, for each point $x\in W_1$ the leaf $\cW^0_x$ intersects $W_2$ and conversely for each point $y\in W_2$ the leaf $\cW^0_y$ intersects $W_1$. We assume that the curves $W_i$ are close enough and not too long, so that the connected pairs $(x,y)\in W_1\times W_2$ are uniquely defined by demanding that the connecting leaf be shorter than a small number $\ell_0<1$. Then the holonomy map $\bfh:W_1\to W_2$ is defined by sliding along the leaf: $\bfh x = y$. Since the images $\cT_n W_i$ are connected by the stable foliation $\cW^n$, one can define the holonomy map $\bfh_n=\cT_n\circ \bfh \circ \cT_n^{-1}:\cT_nW_1\to \cT_nW_2$.

\begin{remark}
Notice that if the curves $W_i$ carry measures $\nu_i$ that are compatible in the sense that $\nu_2 = \bfh\nu_1$, then the images $\cT_n W_i$ carry compatible measures: $\cT_n\nu_2 = \bfh_n \cT_n\nu_1$. This will guarantee in the following that once some of the masses on two points have been coupled to each other, they remain coupled.
\end{remark}

The holonomy map $\bfh$ is said to be absolutely continuous, if the measure $\bfh^{-1}m_{W_2}$ is absolutely continuous with respect to the measure $m_{W_1}$. In this case the Jacobian, which measures distortion under the holonomy map, is defined as the Radon--Nikodym derivative $\cJ \bfh= \frac{d (\bfh^{-1}m_{W_2})}{dm_{W_1}}$. The  change-of-variables formula for integrals is $dm_{W_2}(y) = \cJ\bfh(x)dm_{W_1}(x)$ with $y=\bfh x$ \footnote{Given a Borel set $A\subset W_1$, we have $\int_{\bfh A}dm_{W_2}=(\bfh^{-1}m_{W_2})(A) = \int_A \frac{d (\bfh^{-1}m_{W_2})}{dm_{W_1}}\,dm_{W_1}$.}. For any $x\in W_1$, we have $\bfh x=\cT_n^{-1}\bfh_n\cT_n x$. It is elementary to check that if $\bfh_n$ is absolutely continuous, then $\bfh$ inherits this property via the identity
\beq\label{eq:Jac_decomp}
\begin{split}
\cJ \bfh(x) = \cJ_{\bfh_n\cT_nW_1}\cT_n^{-1}(\bfh_n \cT_n x)\cdot\cJ \bfh_n(\cT_n x) \cdot \cJ_{W_1}\cT_n(x)
= \frac{ \cJ_{W_1}\cT_n(x)}{\cJ_{W_2}\cT_n(\bfh x)} \cdot\cJ \bfh_n(\cT_n x). 
\end{split}
\eeq
By reversing the argument, we see that if $\bfh_m$ is absolutely continuous for some $m$, then $\bfh_n$ is absolutely continuous and \eqref{eq:Jac_decomp} holds for all values of $n\geq 0$.

\begin{lem}[Absolute continuity of the holonomy map]\label{lem:holonomy_Jac_bound}
Let $\bfh$ be as above. It is absolutely continuous. Moreover, there exist constants $c_1\geq 1$ and $0<\mu<1$, independent of the curves $W_1$ and $W_2$ and the choice of the sequence $(T_i)_{i=1}^{n_Q}$, such that
\beq\label{eq:holonomy_Jac_bound}
 \left|\ln \cJ \bfh_n(\cT_n x)  \right|  \leq c_1\mu^n
\eeq
holds for $x\in W_1$ and $0\leq n\leq n_Q$. In particular, $e^{-c_1}\leq \cJ\bfh\leq e^{c_1}$.
\end{lem}
As a curiosity, \eqref{eq:holonomy_Jac_bound} continues to hold for $n>n_Q$ since $T_n = \widetilde T_Q$. In particular, the precise value of $\cJ\bfh(x)$ could be obtained as the limit $\lim_{n\to\infty}\frac{ \cJ_{W_1}\cT_n(x)}{\cJ_{W_2}\cT_n(\bfh x)}$. However, we only care about $n\leq n_Q$. What is important above is that $c_1$ and $\mu$ do not change when the lengths of the intervals $I_q$ in \eqref{eq:time} and hence the value of $n_Q$ are increased arbitrarily.
\begin{proof}[Proof of Lemma~\ref{lem:holonomy_Jac_bound}]
Denote $x^n = \cT_n x$ and $y=\bfh x$ for $x\in W_1$, and $W_i^n = \cT_n W_i$ for $i=1,2$.  We also write $y^n=\cT_n y=\bfh_n x^n$. 
Since $T_n = \widetilde T_Q$ and $\cW^{n-1}_x=\cW^s_{Q,x}$ for all $x$, for all $n>n_Q$, we know the following: $\bfh$ inherits absolute continuity from $\bfh_{n_Q}$, \eqref{eq:Jac_decomp} holds for all $n\geq 0$ as explained above, and $\lim_{n\to\infty}\cJ\bfh_n(x^n) = 1$. Therefore, with the aid of the chain rule $\cJ_W \cT_{n}(x) = \cJ_{W^{n-1}} T_n(x^{n-1})\cdots \cJ_{W^0} T_1(x^0)$, we conclude that
\beqn
\cJ \bfh_m(x^m) =  \prod_{n\geq m} \frac{\cJ_{W_1^n} T_{n+1} (x^n)}{\cJ_{W_2^n}T_{n+1}(y^n) } .
\eeqn
By Assumption (A3), we may use Lemma~\ref{lem:coalesce} on each of the time intervals $I_q$. Since $|\ln z| \leq \max(z-1,z^{-1}-1)$ for all $z>0$, we see that \eqref{eq:coalesce} implies \eqref{eq:holonomy_Jac_bound}:
\beqn
\begin{split}
 \left|\ln \cJ \bfh_m(x^m)  \right|  \leq \sum_{n\geq m}\left| \ln \frac{\cJ_{W_1^n} T_{n+1} (x^n)}{\cJ_{W_2^n}T_{n+1}(y^n) } \right| 
\leq C' \frac{\mu^m}{1-\mu} = c_1\mu^m.
\end{split}
\eeqn
\end{proof}

\begin{lem}[Regularity of the holonomy map]\label{lem:holonomy_reg} 
There exist $0<\eta_\bfh<1$ and $C_\bfh>0$, such that the following holds. Let $W_1$ and $W_2$ be standard unstable curves connected by the stable foliation $
\cW^0$ as above. For $x_1,x_2\in W_1$ such that $|W_1(x_1,x_2)|\leq 1$,
\beqn
|\ln \cJ \bfh(x_1) - \ln \cJ \bfh(x_2)| \leq C_\bfh |W_1(x_1,x_2)|^{\eta_\bfh}.
\eeqn
\end{lem}
\begin{proof}
Denote $x^n = \cT_n x$ for all $x$ in $W_1^n = \cT_n W_1$. We also set $y_i=\bfh x_i$, $y_i^n = \cT_n y_i$, and $W_2^n = \cT_n W_2$. By \eqref{eq:Jac_decomp},
\beqn
\begin{split}
|\ln \cJ \bfh(x_1) - \ln \cJ \bfh(x_2) |
\leq  \left| \ln \frac{\cJ_{W_1}\cT_m(x_1)}{\cJ_{W_1}\cT_m(x_2)} \right| + \left| \ln \frac{\cJ_{W_2}\cT_m(y_2)}{\cJ_{W_2}\cT_m(y_1)} \right|
+ \sum_{i=1,2} \! \left|\ln \cJ \bfh_m(x_i^m) \right|,
\end{split}
\eeqn
for all $m$. The last sum can be bounded with the aid of \eqref{eq:holonomy_Jac_bound}. Using the chain rule
$
\cJ_W \cT_{n}(x) = \cJ_{W^{n-1}} T_n(x^{n-1})\cdots \cJ_{W^0} T_1(x^0),
$
Lemma~\ref{lem:distortion_loc} and the bounds \eqref{eq:growth_bound},
\beqn
\begin{split}
\left| \ln \frac{\cJ_{W_1}\cT_{n_Q}(x_1)}{\cJ_{W_1}\cT_{n_Q}(x_2)} \right| & \leq \sum_{0\leq q\leq Q-1} \left| \ln \frac{\cJ_{W_1^{n_q}}\cT_{n_{q+1},n_{q}+1}(x_1^{n_q})}{\cJ_{W_1^{n_q}}\cT_{n_{q+1},n_{q}+1}(x_2^{n_q})} \right| 
\\
& \leq \sum_{0\leq q\leq Q-1} C_{\mathrm{d},q+1} \! \left| \cT_{n_{q+1},n_{q}+1} \!\left(W_1^{n_q}(x_1^{n_q},x_2^{n_q})\right) \right| 
\\
& \leq \sum_{0\leq q\leq Q-1} C_{\mathrm{d},q+1} \bar\Lambda_{q+1}^{n_{q+1}-n_q} \left |W_1^{n_q}(x_1^{n_q},x_2^{n_q})\right|
\\
& \leq \sum_{0\leq q\leq Q-1} C_{\mathrm{d},q+1} \bar\Lambda_{q+1}^{n_{q+1}-n_q}\cdots \bar\Lambda_2^{n_2-n_1}\bar \Lambda_1^{n_1}  |W_1(x_1,x_2)|
\\
& \leq
\left(\max_{1\leq q\leq Q} C_{\mathrm{d},q}\right) \! |W_1(x_1,x_2)| \sum_{1\leq q\leq Q}\left(\max_{1\leq q\leq Q} \bar\Lambda_{q}\right)^{n_q}
\\
& \leq
\frac{\max_{1\leq q\leq Q} C_{\mathrm{d},q}}{1-\left(\max_{1\leq q\leq Q} \bar\Lambda_{q}\right)^{-1}} \left(\max_{1\leq q\leq Q} \bar\Lambda_{q}\right)^{n_Q} |W_1(x_1,x_2)| 
\\
& = C\bar\Lambda^{n_Q} |W_1(x_1,x_2)|.
\end{split}
\eeqn
Similarly, for any $m$,
\beqn
\left| \ln \frac{\cJ_{W_1}\cT_{m}(x_1)}{\cJ_{W_1}\cT_{m}(x_2)} \right| \leq  C\bar\Lambda^{m} |W_1(x_1,x_2)|
\quad\text{and}\quad
 \left| \ln \frac{\cJ_{W_2}\cT_{m}(y_2)}{\cJ_{W_2}\cT_{m}(y_1)} \right| \leq C\bar\Lambda^{m} |W_2(y_1,y_2)|.
\eeqn
Notice from the definition of $C$ and $\bar\Lambda$ that they are independent of the curves $W_i$ and of the sequence $(T_i)$. Because $| W_2 (y_1,y_2) | \leq \sup_{W_1 (x_1,x_2)}\cJ \bfh \cdot |W_1 (x_1,x_2)|$, 
\beqn
\left| \ln \frac{\cJ_{W_1}\cT_{m}(x_1)}{\cJ_{W_1}\cT_{m}(x_2)} \right| +  \left| \ln \frac{\cJ_{W_2}\cT_{m}(y_2)}{\cJ_{W_2}\cT_{m}(y_1)} \right|  \leq C\bar\Lambda^{m} ( 1+ e^{c_1}) |W_1(x_1,x_2)|  = c_2 \bar\Lambda^m  |W_1(x_1,x_2)|.
\eeqn

Finally, choose $m = \frac{\ln|W_1(x_1,x_2)|}{\ln\mu}$. Then $\mu^m = |W_1(x_1,x_2)|$, $\bar\Lambda^m = |W_1(x_1,x_2)|^{\ln\bar\Lambda/\ln\mu}$, and
\beqn
|\ln \cJ \bfh(x_1) - \ln \cJ \bfh(x_2) | \leq  2c_1\mu^m + c_2 \bar\Lambda^m  |W_1(x_1,x_2)| \leq (2c_1+c_2) |W_1(x_1,x_2)|^{1-\ln\bar\Lambda/|\ln\mu |}.
\eeqn
\end{proof}


\section{Coupling Lemma and the proof of Theorems~\ref{thm:equidistr} and~\ref{thm:weak_conv}}\label{sec:coupling}
Let $(W,\nu)$ be a standard pair and $d\nu=\rho \,dm_W$. We will be interested in densities of the form $\tau \rho$ where $\tau:W\to[0,1]$ is a function. These can be considered as  portions of the measure $\nu$. In practice, we will replace $W$ by the rectangle $\hat W = W\times [0,1]$ with base $W$ and $d\nu$ by the measure $d\hat\nu=d\nu\otimes dt$, where $dt$ denotes the Lebesgue measure on $[0,1]$, and look at the subdomain $\{(x,t)\in \hat W\,:\,0\leq t\leq \tau(x)\}$ of $\hat W$. Introducing the rectangle facilitates bookkeeping.

A standard family $\cG=\{(W_\alpha,\nu_\alpha)\}_{\alpha\in\fA}$ can similarly be replaced by $\hat\cG=\{(\hat W_\alpha,\hat \nu_\alpha)\}_{\alpha\in\fA}$. The measure $\mu_\cG$ induces canonically a measure $\hat\mu_\cG$ on $\cup_\alpha \hat W_\alpha$. A map $T$ on $\cM$ extends to a map on $\cM\times [0,1]$ by setting $T(x,t) \equiv (T(x),t)$ and all observables on $\cM$ extend to observables on $\cM\times [0,1]$ by setting $f(x,t) \equiv f(x)$.

We are now in position to state the following key result. 
\begin{lem}[Coupling Lemma]\label{lem:coupling}
Consider two standard families $\cG=\{(W_\alpha,\nu_\alpha)\}_{\alpha\in\fA}$ and $\cE=\{(W_\beta,\nu_\beta)\}_{\beta\in\fB}$. There exist an almost everywhere defined bijective map $\Theta:\cup_\alpha \hat W_\alpha \to \cup_\beta \hat W_\beta$, called the coupling map, that preserves measure, \ie, $\Theta(\hat\mu_\cG)=\hat\mu_\cE$, and an almost everywhere defined function $\Upsilon:\cup_\alpha \hat W_\alpha\to\bN$, called the coupling time, both depending on the sequence $(T_i)$, such that the following hold:
\begin{enumerate}
\item Let $(x,t)\in \hat W_\alpha$, $\alpha\in\fA$, and $\Theta(x,t)=(y,s)\in \hat W_\beta$, $\beta\in\fB$. Then the points $x$ and $y$ lie on the same leaf, say $W$, of the stable foliation $\cW^0$. If $n\geq \Upsilon(x,t)$, then the distance of the points $\cT_n x$ and $\cT_n y$ along the leaf $\cT_n W$ of the stable foliation $\cW^n$ satisfies $|\cT_nW(\cT_n x,\cT_n y)|<\ell_0\lambda^{n-\Upsilon(x,t)}$. Here $\ell_0>0$ has been introduced earlier and $\lambda=\max_{1\leq q\leq Q}\Lambda_q^{-1} <1$.
\item The exponential tail bound
\beq\label{eq:tailbound}
\hat\mu_{\cG}(\Upsilon>n) \leq C_\Upsilon\vartheta_\Upsilon^n
\eeq
holds for uniform constants $C_\Upsilon>0$ and $\vartheta_\Upsilon\in(0,1)$. 
\end{enumerate}
\end{lem}

\begin{proof}[Proof of Theorem~\ref{thm:equidistr}]
We use the coupling between $\cG$ and $\cE$ given in the Coupling Lemma:
\beqn
\begin{split}
& \int_{\cM} f\circ  \cT_n\,d\mu_{\cG} - \int_{\cM} f\circ  \cT_n\,d\mu_{\cE} \\
& \qquad\qquad =  \int_{\cM\times [0,1]} (f\circ  \cT_n)(x,t)\,d\hat\mu_{\cG}(x,t) - \int_{\cM\times [0,1]} (f\circ  \cT_n)(y,s)\,d\hat\mu_{\cE}(y,s) \\
& \qquad\qquad= \int_{\cM\times [0,1]} (f\circ  \cT_n)(x,t)\,d\hat\mu_{\cG}(x,t) - \int_{\cM\times [0,1]} (f\circ  \cT_n\circ \Theta)(x,t)\,d\hat \mu_{\cG}(x,t) \\
& \qquad\qquad= \int_{\cM\times [0,1]} \left( f\circ  \cT_n-f\circ  \cT_n\circ \Theta\right)\,d\hat \mu_{\cG} \\
& \qquad\qquad= \int_{\Upsilon\leq n/2} \left( f\circ  \cT_n-f\circ  \cT_n\circ \Theta\right)\,d\hat \mu_{\cG} + \int_{\Upsilon> n/2} \left( f\circ  \cT_n-f\circ  \cT_n\circ \Theta\right)\,d\hat \mu_{\cG}.
\end{split}
\eeqn
By \eqref{eq:tailbound},
\beqn
\left|\int_{\Upsilon> n/2} \left( f\circ  \cT_n-f\circ  \cT_n\circ \Theta\right)\,d\hat \mu_{\cG}\right| \leq C_\Upsilon(\sup f-\inf f) \vartheta_\Upsilon^{n/2} .
\eeqn
On the other hand, assume $\Upsilon(x,t)\leq n/2$. Then $|(f\circ \cT_n-f\circ \cT_n\circ \Theta)(x,t)| \leq {{|f|}_\gamma} (\ell_0\lambda^{n-\Upsilon(x,t)})^\gamma$, by the Coupling Lemma, such that
\beqn
\left|\int_{\Upsilon\leq n/2} \left( f\circ \cT_n-f\circ \cT_n\circ \Theta\right)\,d\hat \mu_{\cG}\right| \leq  \ell_0^\gamma {|f|}_\gamma \lambda^{\gamma n/2}.
\eeqn
Since $\ell_0<1$, the proof is complete.
\end{proof}

\begin{proof}[Proof of Theorem~\ref{thm:weak_conv}]
First notice that both of the measures $\mu^i$ can be disintegrated using a suitable measurable partition of $\cM$ so that we almost obtain two standard families, with the nuisance that the H\"older constants of the logarithms of the conditional measures possibly exceed $C_{\mathrm{r}}$ in \eqref{eq:regular_density}. In the latter case we need a finite waiting time $N=N(\rho^1,\rho^2)$, depending on the H\"older constants of $\ln\rho^i$, until the densities regularize and yield true standard families; see the proof of Lemma~\ref{lem:standard_image}.  For $\gamma$-H\"older observables the result then follows immediately from Theorem~\ref{thm:equidistr}, with the above waiting time giving the constant $C_\gamma(\rho^1,\rho^2)=\theta_\gamma^{-N}$. If $f$ is only continuous, we fix an arbitrarily small $\ve>0$ and, by Stone--Weierstrass theorem, pick a $\gamma$-H\"older $f_\ve$ such that $\|f-f_\ve\|_\infty<\ve$. Then $\left|\int_{\cM} f\circ \cT_n\, d\mu^1 - \int_{\cM} f\circ \cT_n\, d\mu^2\right| < C_\gamma(\rho^1,\rho^2)B_{f_\ve}\theta_\gamma^n+2\ve<3\ve$ if $n>\ln(\ve/C_\gamma(\rho^1,\rho^2)B_{f_\ve})/\ln\theta_\gamma$.
\end{proof}

\section{Proof of the Coupling Lemma}\label{sec:coupling_proof}

\subsection{Outline of the proof}
The idea of the proof is to construct special tiny rectangles, called magnets, which can be thought to attract unstable curves. Mixing guarantees that a small fraction, say 1 percent, of any high enough iterate of any unstable curve will ultimately lie on a magnet. Once two unstable curves from two different standard families cross a magnet, we are able to couple a fraction of their masses by connecting some of their points lying on the magnet with very short stable manifolds. This has to be done with due care, because the resulting coupling has to be measure preserving. 

The process is then repeated recursively, and so the construction of the coupling map $\Theta$ and the coupling time function $\Upsilon$ is recursive. It can be shown that after a fixed finite number of iterates a fixed fraction of the remaining masses can always be coupled, so that the measures on the unstable curves can be `drained' at an exponential rate. 

Since we are dealing with compositions of diffeomorphisms from the sequence $(T_i)$ rather than iterates of a single diffeomorphism, we need to use time-dependent magnets. For $n\in I_q$, $T_n\in \cU_q$, and the magnet to be used should reflect the structure of the reference diffeomorphism $\widetilde T_q$. In our time-dependent, finite time, setting it is not even \emph{a priori} clear what coupling should mean. We choose to construct a coupling via the stable foliations $\cW^n$ that vary from one point in time to the next. As mentioned earlier, these foliations are artificial in the sense that they depend on the artificial future $T_n = \widetilde T_Q$ for times $n> n_Q$ although in reality we only consider the compositions $\cT_n=T_n\circ\dots\circ T_1$ for $n\leq n_Q$. We then have to pay special attention to uniformity: our convergence rates, \textit{etc.}, should depend neither on the particular value of $n_Q$ nor on the particular finite sequence $(T_1,\dots,T_{n_Q})$ as long as the reference automorphisms $(\widetilde T_1,\dots, \widetilde T_Q)$ have been chosen and the earlier assumptions on $(T_1,\dots,T_{n_Q})$ are being respected. 

\subsection{Magnets and crossings} 
In this subsection $1\leq q\leq Q$ is fixed for good. Unstable curves and standard pairs are to be understood as being defined with respect to the cone family  $\{\cC^u_{q,x}\}$ with $q$ fixed.

Consider the Anosov diffeomorphism $\widetilde T_q$.
A `rectangle', $\fR\subset \cM$, is a closed and connected region bounded by two stable manifolds and two unstable manifolds of $\widetilde T_q$. These are called the \emph{s- and u-sides} of the rectangle, respectively. Recalling that $\cW^s_{q,x}$ denotes the stable leaf of $\widetilde T_q$ at~$x$, we also assume that the size of the rectangle in the stable direction satisfies $|\cW^s_{q,x}\cap\fR|\ll \ell_0$.   

We say that an unstable curve $W$ crosses the rectangle \emph{properly}, if
\begin{itemize}
\item[(P1)] $W$ crosses $\fR$ completely, \ie, $W\cap \fR$ contains a connected curve $W'$ connecting the two s-sides of the rectangle, and
\item[(P2)] both components of $W\setminus W'$ are of length strictly greater than $\ell/10$,
\end{itemize}
both hold. In other words, a crossing is proper if the curve crosses the rectangle completely and there is a guaranteed amount of excess length beyond each s-side of the rectangle. Here $\ell$ is the lower bound on the length of a standard curve. 
Finally, an unstable curve $W$ crosses the rectangle \emph{super-properly}, if (P1),
\begin{itemize}
\item[(P2')] both components of $W\setminus W'$ are of length strictly greater than $\ell/5$, and
\item[(P3)] each $x\in W\cap \fR$ divides the curve $\cW^s_{q,x}\cap \fR$ in a ratio strictly between $1/10$ and $9/10$.
\end{itemize}
all hold. Thus, in a super-proper crossing there is more guaranteed excess length than in a proper crossing and the curve also stays well clear of the u-sides. 
\begin{lem}\label{lem:rectangles} 
There exists a finite set of rectangles, $\{\fR^k\,:\,1\leq k\leq k_0\}$, such that each standard unstable curve crosses at least one of the rectangles super-properly.
\end{lem}
\begin{proof}
Every closed standard curve crosses some rectangle $\fR$ super-properly. Since crossing a rectangle $\fR$ super-properly is an open condition in the Hausdorff metric, the set $\cU_\fR$ of all closed standard curves crossing $\fR$ super-properly is an open set. The collection formed by all the sets $\cU_\fR$ is an open cover of the space of closed standard curves equipped with the Hausdorff metric. The latter space is compact. We can therefore pick a finite subcover and the corresponding rectangles. 
\end{proof}
We now pick arbitrarily one of the rectangles $\fR^k$. This special rectangle, that we will denote by $\fR_q$, will be called a \emph{magnet}. It will serve as a reference set on which points will be coupled. 

\begin{lem}\label{lem:super_to_proper}
Fix $n\geq 1$. By taking $\ve_q$ sufficiently small (depending on $n$) the following holds. If $W$ is an unstable curve and $\widetilde T_q^n W$ crosses $\fR_q$ super-properly, then $\cT_n W$ crosses $\fR_q$ properly, provided each $T_i\in \cU_q$.
\end{lem}
\begin{proof}
For any point $x$, we have the bound $d(\cT_n x,\widetilde T_q^n x)\leq C(n)\ve_q$.
\end{proof}

If $W$ is an unstable curve and $n$ is fixed, let $W^q_{n,i}$, $i\in\fI$, be the connected components of $\widetilde T_q^n W \cap \fR_q$, that correspond to super-proper crossings. That is, each $W^q_{n,i}$ is a subset of a longer curve $\widetilde W^q_{n,i}\subset \widetilde T_q^n W$ which crosses $\fR_q$ super-properly and $\widetilde W^q_{n,i}\cap \fR_q = W^q_{n,i}$.

\begin{lem}\label{lem:rect_intersect_rect}
There exist a subrectangle $\fB_q\subset \fR_q$ and a number $s'\geq 1$ such that the following holds. Assume that $W$ is an unstable curve that crosses a rectangle $\fR^k$ properly, $n\geq s'$, and $\widetilde T_q^n \fR^k\cap \fB_q\neq \varnothing$. Every component of $\widetilde T_q^n \fR^k\cap\fR_q$ that intersects $\fB_q$ intersects $W^q_{n,i}$ for precisely one value of the index $i\in \fI$.
\end{lem}

In words, each intersection of $\widetilde T_q^n \fR^k$ with $\fB_q$ yields a super-proper crossing of $\widetilde T_q^n W$, as long as $n$ is large enough.
\begin{proof}
We  assume that the magnet $\fR_q$ is so small that the leaves of the unstable foliation of $\widetilde T_q$ are almost parallel lines on $\fR_q$. This can be guaranteed by considering only sufficiently small rectangles in the proof of Lemma~\ref{lem:rectangles}. Now, choose $\fB_q\subset \fR_q$ to be a rectangle whose distance to the u-sides of $\fR_q$ is sufficiently large; say each $x\in \fB_q$ divides the curve $\cW^s_{q,x}\cap \fR_q$ in a ratio between $1/5$ and $4/5$. As $\widetilde T_q$ is one-to-one, the components of $\widetilde T_q^n \fR^k\cap \fR_q$ are disjoint. Assuming $s'$ is large, these components are very thin strips, almost aligned with the unstable foliation. Pick such a component and assume that it intersects $\fB_q$. It is a safe distance away from the u-sides of the magnet. Inside this component lies a piece $V$ of the curve $\widetilde T_q^n W$. The piece $V$ has to extend to a super-proper crossing of $\fR_q$, because $W$ crosses $\fR^k$ properly and because $n_1$ is large. Thus, $V$ is actually a subcurve of one of the $W^q_{n,i}$.
\end{proof}

\begin{lem}\label{lem:sp_crossing}
There exist numbers $d''>0$ and $s''\geq 1$ such that if $(W,\nu)$ is a standard pair and $n\geq s''$, then $\nu(\widetilde T_q^{-n}(\cup_{i\in\fI} W^q_{n,i}))\geq d''$.
\end{lem}
In other words, the fraction of $W$ that will cross the magnet $\fR_q$ super-properly after $n$ steps is at least $d''$.

\begin{proof}
Fix a $k$ such that $W$ crosses $\fR^k$ properly. This is possible by Lemma~\ref{lem:rectangles}. By Lemma~\ref{lem:rect_intersect_rect}, if a component of $\widetilde T_q^n \fR^k\cap\fR_q$ intersects the subrectangle $\fB_q$, it is crossed by the curve component $W^q_{n,i}$ for precisely one value of the index $i\in\fI$. In this case let $\fR^{k,q}_{n,i}$ denote the former component of $\widetilde T_q^n \fR^k\cap \fR_q$. Thus  $\fR^{k,q}_{n,i}$ is only defined for a subset $\fI_{\fB_q}\subset \fI$ of indices. We have $\nu(\widetilde T_q^{-n} (\cup_{i\in\fI} W^q_{n,i})) \geq\nu(\widetilde T_q^{-n} (\cup_{i\in\fI_{\fB_q}} W^q_{n,i}))  = \nu(\widetilde T_q^{-n} (\cup_{i\in\fI_{\fB_q}} W^q_{n,i} \cap \fR^{k,q}_{n,i}))  \geq c\mu_q(\widetilde T_q^{-n} (\cup_{i\in\fI_{\fB_q}} \fR^{k,q}_{n,i})) = c\mu_q (\cup_{i\in\fI_{\fB_q}} \fR^{k,q}_{n,i}) \geq c\mu_q (\cup_{i\in\fI_{\fB_q}} \fR^{k,q}_{n,i}\cap \fB_q) = c\mu_q (\widetilde T_q^n \fR^k\cap \fB_q)  \geq \frac{c}{2} \mu_q (\fR^k)\mu_q (\fB_q) $ if $n\geq s''$ and $s''$ is large. The last step in the estimate follows from mixing of the invariant measure $\mu_q$. The third step relies on the absolute continuity with bounded Jacobians of the holonomy maps of $\widetilde T_q$ as well as on the regularity of $\nu$ and of the conditional measures of $\mu_q$ on the unstable leaves of $\widetilde T_q$. To finish, notice that $d''=\mu_q (\fR^k)\mu_q (\fB_q)>0$, since the interiors of $\fR^k$ and $\fB_q$ are nonempty.
\end{proof}

For an unstable curve $W$, let $W_{n,i}$ now be the connected components of $\cT_n W \cap \fR_q$, labeled by $i$, that correspond to proper crossings. That is, each $W_{n,i}$ is a subset of a longer curve $\widetilde W_{n,i}\subset \cT_n W$ which crosses $\fR_q$ properly and $\widetilde W_{n,i}\cap \fR_q = W_{n,i}$.

\begin{cor}\label{cor:sp_crossing}
There exist numbers $d'_0>0$ and $s_0'\geq 1$ such that the following holds. Let $T_i\in \cU_q$ for each $i$ and $(W,\nu)$ be a standard pair. If  $n\geq s_0'$, then $\nu(\cT_n^{-1}(\cup_i W_{n,i}))\geq d'_0$.
\end{cor}
\begin{proof}
Fix $m\geq \ln\frac{2}{C_q}/\ln\Lambda_q$. By Lemma~\ref{lem:standard_image}, $(\cT_m W,\cT_m \nu)$ can be broken into a finite collection of standard pairs $(W_j,\nu_j)$ such that $\cT_m W=\cup_j W_j$ and $\cT_m \nu=\sum_j c_j\nu_j$, where $0<\nu_j<1$ and $\sum_j c_j=1$. For each $j$, by Lemma~\ref{lem:sp_crossing}, there is a finite collection of disjoint (minimal) subcurves $V_{j,k}\subset W_j$ such that $\widetilde T_q^{s''}V_{j,k}$ crosses the magnet $\fR_q$ super-properly. Moreover, $\sum_k\nu_j(V_{j,k})\geq d''$ and, by Lemma~\ref{lem:super_to_proper}, the images $\cT_{m+s'',m+1}V_{j,k}$ cross the magnet $\fR_q$ properly. We also have $\cT_m\nu(\cup_j\cup_k V_{j,k}) = \sum_jc_j\sum_k \nu_j(V_{j,k})\geq d''$. Let us relabel the collection of the subcurves $\cT_m^{-1}V_{j,k}\subset W$ by $U_l$. We have so far shown that $\nu(\cup_l  U_l)\geq d''$ and that $\cT_{m+s''} U_l$ crosses the magnet $\fR_q$ properly. 

We are almost done, but we still need to truncate each $U_l$ to a subcurve $\tilde U_l$ so that $\cT_{m+s''} \tilde U_l = \cT_{m+s''} U_l\cap \fR_q$ and argue that $\nu(\cup_l  \tilde U_l)\geq \alpha\nu(\cup_l U_l)$ for some constant $\alpha>0$. Such a truncation amounts to choosing the subcurve $\tilde V_{j,k}=\cT_m\tilde U_l$ of $V_{j,k}\subset W_j$ so that $\cT_{m+s'',m+1}\tilde V_{j,k} = \cT_{m+s'',m+1}V_{j,k}\cap \fR_q$. Now $\nu_j(\tilde V_{j,k})\geq \alpha\nu_j(V_{j,k})$ follows from two observation: 
\begin{itemize}
\item $|\tilde V_{j,k}|$ is bounded uniformly away from zero, because $\cT_{m+s'',m+1}\tilde V_{j,k}$ crosses $\fR_q$ completely.
\item $|\cT_{m+s'',m+1}V_{j,k}\setminus \fR_q|$ is bounded uniformly from above, because $|\widetilde T_q^{s''}V_{j,k}\setminus \fR_q|$ was assumed to be as small as possible (for a super-properly crossing curve). Therefore $|V_{j,k}\setminus \tilde V_{j,k}|=|V_{j,k}\setminus \cT_{m+s'',m+1}^{-1}\fR_q |$ is bounded uniformly from above.
\end{itemize}
Indeed, it is implied that $|\tilde V_{j,k}|/|V_{j,k}|\geq\alpha'$ for some $\alpha'\in(0,1]$ so that, by estimate \eqref{eq:density_comp}, $\nu_j(\tilde V_{j,k})/\nu_j(V_{j,k}) \geq D^{-1}\alpha'$ for all $j,k$.
 
Notice that only $s''$ affected the size of $\ve_q$. This happened when Lemma~\ref{lem:super_to_proper} was used. 
\end{proof}

\subsection{Time-dependent magnets}
For the rest of the section, let $(T_i)$ be a sequence of the form described in the Introduction, which is not confined to a neighborhood $\cU_q$ of any one map $\widetilde T_q$.

The Coupling Lemma needs to hold for the compositions in \eqref{eq:composition}. For this reason we cannot use the same magnet for all times. Moreover, the stable foliation $\cW^n$ that we will use to couple points (more correctly some of the probability masses carried by these points) on the magnets changes with time. This will guarantee that what has aleady been coupled will always remain coupled.  Therefore, we need to introduce the following time-dependent magnets:
For every $q\in\{1,\dots,Q\}$ and every $n\in I_q$, define  
\beqn
\fM_n = \{ x\in\fR_q \,:\; \text{$\cW^n_x\cap\fR_q$ connects the u-sides of $\fR_q$ and has only one component} \}.
\eeqn
In other words, $\fM_n$ consists of those leaves of $\cW^n$ that connect the u-sides of $\fR_q$ \emph{and} are entirely inside $\fR_q$. 

We say that an unstable curve $W$ crosses $\fM_n$ properly if  $n\in I_q$ and $W$ crosses $\fR_q$ properly. Assuming that the cones $\cC^s_{q,x}$ are narrow enough, $\fM_n$ is close to $\fR_q$, and we have $|W\cap\fM_n|\geq \frac34|W\cap\fR_q|$. 

\begin{lem}\label{lem:mini_magnet}
Assume that $W$ crosses the magnet $\fR_q$ properly and that $W$ carries the measure $d\nu = \rho\,dm_W$, where $\rho$ satisfies \eqref{eq:regular_density}. Then, for all $n\in I_q$,
\beqn
\nu(W\cap\fM_n)\geq \frac12\nu(W\cap\fR_q).
\eeqn
\end{lem}
\begin{proof}
By \eqref{eq:regular_density}, $\rho(y)\geq \rho(x)e^{-C_\mathrm{r}|W\cap\fR_q|^{\eta_\mathrm{r}}}$ for any $x,y\in W\cap\fR_q$. In particular, we can average the left side over $W\cap\fM_n$ and the right side over $W\cap\fR_q$, obtaining $\nu(W\cap\fM_n) \geq \frac{|W\cap\fM_n|}{|W\cap\fR_q|} e^{-C_\mathrm{r}|W\cap\fR_q|^{\eta_\mathrm{r}}} \nu(W\cap\fR_q)$. Since $\fR_q$ has small diameter, the exponential factor is close to $1$. For  $n\in I_q$, we also have $|W\cap\fM_n|\geq \frac34|W\cap\fR_q|$.
\end{proof}

For a standard family with respect to $\{\cC^u_{1,x}\}$, $\cG=\{(W_\alpha,\nu_\alpha)\}_{\alpha\in\fA}$, let $W_{\alpha,n,i}$ be the connected components of $\cT_n W_\alpha \cap \fM_n$ that correspond to proper crossings, and introduce the notation 
\beqn
W_{\alpha,n,\star}=\cT_n^{-1}(\cup_i W_{\alpha,n,i}).
\eeqn

Next, we generalize Corollary~\ref{cor:sp_crossing} of Lemma~\ref{lem:sp_crossing}.
\begin{lem}\label{lem:sf_crossing}
There exist numbers $s_0\geq 1$ and $d_0>0$, such that the following holds.
If $(T_i)$ is a sequence of the general form described in the Introduction and $\cG=\{(W_\alpha,\nu_\alpha)\}_{\alpha\in\fA}$ a standard family with respect to $\{\cC^u_{1,x}\}$ then, for $1\leq q\leq Q$ and $n_{q-1}+s_0\leq n\leq n_q$,
\beqn
\mu_\cG(\cup_\alpha W_{\alpha,n,\star})=\int_\fA \nu_\alpha(W_{\alpha,n,\star}) \,d\lambda_\cG(\alpha)\geq d_0.
\eeqn 
\end{lem}
In other words, there are time windows for  $n$, such that a significant fraction of the image under $\cT_n$ of the standard family $\cG$ lies on the magnet $\fM_{n}$, ready to be coupled. 
\begin{remark}
Fixing some $1\leq q\leq Q$ and $k\in I_q$, Lemma~\ref{lem:sf_crossing} can be applied to the shifted sequence $(T_i)_{i\geq k}$ and $\cG=\{(W_\alpha,\nu_\alpha)\}_{\alpha\in\fA}$ a standard family with respect to $\{\cC^u_{q,x}\}$. Then, if $k-1+s_0\leq n\leq n_{q}$, as well as if $n_{q'-1}+s_0\leq n\leq n_{q'}$ for some $q'\in\{q+1,\dots, Q\}$, at least a $d_0$-fraction of its image under $\cT_{n,k}$ lies on the magnet $\fM_n$ as a result of proper crossings. 
\end{remark}
\begin{proof}[Proof of Lemma~\ref{lem:sf_crossing}]
The $s_0'$ in Corollary~\ref{cor:sp_crossing} depends on $q$. We take $s_0$ larger than the maximum of these numbers over $1\leq q\leq Q$. If $n_{q-1}+s_0\leq n\leq n_q$, and if $s_0$ is taken sufficiently larger than $s_0'$, then the image of $\cG$ under $\cT_{n_{q-1}+s_0-s_0'}$ after standardizing the curves becomes a standard family with respect to $\{\cC^u_{q,x}\}$. Applying Corollary~\ref{cor:sp_crossing} to this standard family and the map $\cT_{n,n_{q-1}+s_0-s_0'+1}$ yields a lower bound on the $\cT_n\mu_\cG$-measure of proper crossings of $\fR_q$. From this we infer a lower bound on the proper crossings of $\fM_n$ by the regularity of densities with the aid of Lemma~\ref{lem:mini_magnet}.
\end{proof}

\subsection{Coupling step}
Consider first two standard families, $\cG = \{(W_\alpha,\nu_\alpha)\}$ and $\cE=\{(W_\beta,\nu_\beta)\}$, consisting of one standard pair each. 

A good fraction of the images $\cT_{s_0}W_\alpha$ and $\cT_{s_0}W_\beta$ cross the magnet $\fM_{s_0}$ properly, so that $\cT_{s_0}\nu_\alpha(\cup_i W_{\alpha,s_0,i})=\nu_\alpha(W_{\alpha,s_0,\star} ) > d_0$ and $\cT_{s_0}\nu_\beta(\cup_j W_{\beta,s_0,j})=\nu_\beta(W_{\beta,s_0,\star} ) > d_0$. Here $i$ and $j$ run through some finite index sets and $\cT_{s_0}\nu$ is the pushforward of $\nu$. 

Recall that, for a curve $W$, $\hat W$ denotes the rectangle $W\times [0,1]$ with base $W$, and if $W$ carries a measure $d\nu$ then $\hat W$ carries the measure $d\hat\nu=d\nu\otimes dt$. We will construct a coupling from a subset of $\cup_i \hat W_{\alpha,s_0,i}$ to a subset of $\cup_j \hat W_{\beta,s_0,j}$ and then show that the complements of these subsets can be coupled recursively.

With small preliminary preparations, we can assume that the cardinalities of the index sets for $i$ and $j$ are the same, so that we can pair each $\hat W_{\alpha,s_0,i}$ with precisely one $\hat W_{\beta,s_0,i}$. Furthermore, we can assume that their relative masses agree: 
\beq\label{eq:rel_mass}
\frac{\hat\nu_{\alpha,s_0,i}(\hat W_{\alpha,s_0,i})}{Z_{\alpha,s_0}}=\frac{\hat\nu_{\beta,s_0,i}(\hat W_{\beta,s_0,i})}{Z_{\beta,s_0}}  \qquad \forall\,i. 
\eeq
Here $\hat \nu _{\slot,s_0,i}$ is the measure on $\hat W_{\slot,s_0,i}$ and $Z_{\slot,s_0}=\sum_l\hat\nu_{\slot,s_0,l}(\hat W_{\slot,s_0,l}) = \nu_{\slot}(W_{\slot,s_0,\star}) $. To see that no generality is lost making such assumptions, consider a rectangle $\hat W$ with a measure $\hat \nu$ on it. We can subdivide it into lower rectangles $W\times I_k$, where $I_k\subset [0,1]$ is an interval. Then, each $W\times I_k$ is stretched affinely onto $W\times [0,1]$ and equipped with the measure $d\hat\nu_k=|I_k| \,d\hat\nu$. In other words, we end up with replicas of $\hat W$ equipped with lowered measures. Such an operation on $\hat W$ is measure preserving, because the pushforward of the measure $\hat\nu|_{W\times I_k}$ under the affine map $\bA:W\times I_k\to W\times [0,1]$ is precisely $\hat \nu_k$. Subdividing the rectangles in the families $\{\hat W_{\alpha,s_0,i}\}$ and $\{\hat W_{\beta,s_0,j}\}$ as necessary, and relabeling the resulting rectangles, we can tune the number of rectangles as well as their relative weights so as to arrive at the convenient situation described above. Each rectangle $\hat W_{\slot,s_0,i}$ now comes with an associated affine map $\bA_{\slot,s_0,i}$ (which is the identity if no subdivision of the particular rectangle was necessary). Some of the rectangles will have a common curve as their base on the manifold $\cM$, but this is not a matter of concern.

For each fixed $i$, we can couple a subset of $\hat W_{\alpha,s_0,i}$ to a subset of $\hat W_{\beta,s_0,i}$ as follows. Choose a number  $\tau_\alpha\in(0,1/2]$ such that 
\beq\label{eq:tau_alpha}
\tau_\alpha\cdot Z_{\alpha,s_0} = \frac{d_0}{2}.
\eeq
Now, fix $i$. Omitting some ornaments for the sake of readability, let $\bfh$ stand for the holonomy map from $W_{\alpha,s_0,i}$ to $W_{\beta,s_0,i}$ associated with the stable foliation $\cW^{s_0}$ and denote by $\rho_\slot$ the density of $\hat\nu_{\slot,s_0,i}$ with respect to $dm_{W_{\slot,s_0,i}}\otimes dt$. 

The subset $\hat W_{\alpha,s_0,i}'=\{(x,t)\in \hat W_{\alpha,s_0,i} \,:\, 0\leq t\leq \tau_\alpha \}$ is coupled to a corresponding subset $\hat W_{\beta,s_0,i}'=\{(y,s)\in \hat W_{\beta,s_0,i} \,:\, 0\leq s\leq \tau_{\beta,i}(y)\}$ via the coupling map $\Theta_{s_0,i}':\hat W_{\alpha,s_0,i}'\to\hat W_{\beta,s_0,i}' : (x,t)\mapsto (y,s)$ with
\beqn
y = \bfh x \qquad\text{and}\qquad s = \frac{\tau_{\beta,i}(y)}{\tau_\alpha} t .
\eeqn
Notice, however, that $\tau_{\beta,i}$ is not constant but a function. It is given by the consistency rule
\beq\label{eq:tau_beta}
\tau_{\beta,i}(y)\rho_\beta(y) = \frac{\tau_\alpha\rho_\alpha(x)}{\cJ \bfh (x)}.
\eeq
The expression on the right-hand side of \eqref{eq:tau_beta} equals the pushforward of the density $\tau_\alpha\rho_\alpha$ under the holonomy map, evaluated at $y=\bfh x$. This guarantees that the coupling is measure preserving:
if $f$ is a measurable function $\hat W_{\beta,s_0,i}'\to \bR$, then
\beqn
\begin{split}
& \int_{\hat W_{\alpha,s_0,i}'} (f\circ \Theta_{s_0,i}')(x,t)\,d\hat \nu_{\alpha,s_0,i}(x,t) = \int_{W_{\alpha,s_0,i}}\int_0^{\tau_\alpha} f(\Theta_{s_0,i}'(x,t))\,\rho_\alpha(x)\,dm_{W_{\alpha,s_0,i}}(x)dt \\
& \quad \qquad = \int_{W_{\alpha,s_0,i}}\left[\int_0^{\tau_\alpha} f(\Theta_{s_0,i}'(x,t))\,dt\right]\rho_\alpha(x)\,dm_{W_{\alpha,s_0,i}}(x) \\
& \quad \qquad = \int_{W_{\beta,s_0,i}}\left[\int_0^{\tau_\alpha} f\!\left(y,\frac{\tau_{\beta,i}(y)}{\tau_\alpha} t \right) dt\right] \frac{\rho_\alpha(x)}{ \cJ \bfh (x)}\,dm_{W_{\beta,s_0,i}}(y) \\
& \quad \qquad = \int_{W_{\beta,s_0,i}}\left[\int_0^{\tau_{\beta,i}(y)} f(y,s) \, ds\right] \frac{\tau_\alpha}{\tau_{\beta,i}(y)}\frac{\rho_\alpha(x)}{ \cJ \bfh (x)}\,dm_{W_{\beta,s_0,i}}(y)
=  \int_{\hat W_{\beta,s_0,i}'} f(y,s)\,d\hat\nu_{\beta,s_0,i}(y,s).
\end{split}
\eeqn

We thus have a coupling for each value of the index $i$, and have therefore managed to couple exactly $d_0/2$ units of mass between the families $\{\hat W_{\alpha,s_0,i}\}$ and $\{\hat W_{\beta,s_0,i}\}$ via a measure preserving map. 

We are now in position to describe the desired coupling map $\Theta$ from a subset $\tilde W_{\alpha}\subset \hat W_{\alpha}$ to a subset $\tilde W_{\beta}\subset \hat W_{\beta}$. Define 
\begin{align*}
\tilde W_{\alpha} &= \{(x,t)\in \hat W_{\alpha} \,:\; \text{$(\cF_{s_0} x,\bA_{\alpha,s_0,i} t)\in  \hat W_{\alpha,s_0,i}'$ for some $i$}\},
\\
\tilde W_{\beta} &=\{(y,s)\in \hat W_{\beta} \,:\; \text{$(\cF_{s_0} y,\bA_{\beta,s_0,i} s)\in \hat W_{\beta,s_0,i}'$ for some $i$}\}.
\end{align*}
The bijective map $\Theta: \tilde W_{\alpha} \to \tilde W_{\beta}$ is defined for a point $(x,t)\in \tilde W_{\alpha}$ such that $(\cF_{s_0} x,\bA_{\alpha,s_0,i} t)\in  \hat W_{\alpha,s_0,i}'$ by the rule
\beqn
(y,s)=\Theta(x,t) \iff (\cF_{s_0} y,\bA_{\beta,s_0,i} s) = \Theta_{s_0,i}'(\cF_{s_0} x,\bA_{\alpha,s_0,i} t).
\eeqn
Because the affine maps and the couplings $\Theta_{s_0,i}'$ are measure preserving, also $\Theta$ is measure preserving; the pushforward of the measure $\hat\mu_\cG|_{\tilde W_\alpha}$ under $\Theta$ is $\hat\mu_\cE|_{\tilde W_\beta}$. In particular, the amount of coupled mass equals $\hat\mu_\cG(\tilde W_\alpha)=\hat\mu_\cE(\tilde W_\beta)=d_0/2$. Finally, the coupling time function $\Upsilon:\tilde W_\alpha\to\bN$ is defined by
\beqn
\Upsilon(x,t)=s_0.
\eeqn

We now have a complete description of how to couple $d_0/2$ units of mass of \emph{any} two standard pairs. Thus, given two standard families  $\cG=\{(W_\alpha,\nu_\alpha)\}_{\alpha\in\fA}$ and $\cE=\{(W_\beta,\nu_\beta)\}_{\beta\in\fB}$, we can couple a subset $\tilde W_\alpha\subset\hat W_\alpha$ with a subset $\tilde W_\beta\subset \hat W_\beta$ for any pair $(\alpha,\beta)\in\fA\times\fB$, in which case we have $\nu_{\alpha}(\tilde W_\alpha)=\nu_\beta(\tilde W_\beta)=d_0/2$. Recall that the index sets $\fA$ and $\fB$ carry probability factor measures $\lambda_\cG$ and $\lambda_\cE$, respectively. They detail how much weight is assigned to standard pairs. Suppose the sets $\fA$ and $\fB$ are finite or countable. Splitting off subrectangles if necessary and stretching them affinely onto complete rectangles as described earlier, we can assume that there exists a bijection $\Delta:\fA\to\fB$ that preserves measure, \ie, $\Delta\lambda_\cG=\lambda_\cE$. Hence, the coupling map $\Theta$ can be constructed from a subset $\cup_{\alpha\in\fA}\tilde W_\alpha\subset \cup_{\alpha\in\fA}\hat W_\alpha$ to a subset $\cup_{\beta\in\fB}\tilde W_\beta\subset \cup_{\beta\in\fB}\hat W_\beta$ so that measure is preserved and in particular so that $\hat\mu_\cG(\cup_{\alpha\in\fA}\tilde W_\alpha)=\hat\mu_\cE(\cup_{\beta\in\fB}\tilde W_\beta)=d_0/2$. On $\cup_{\alpha\in\fA}\tilde W_\alpha$ we set $\Upsilon=s_0$. The map $\Theta$ can be constructed also for uncountable families, but we omit the details \cite{Chernov-BilliardsCoupling}.

\subsection{Recovery step}

Our task is to couple the remaining points of $\hat W_\alpha\setminus \tilde W_\alpha$ to those of $\hat W_\beta\setminus \tilde W_\beta$ and to extend $\Theta$ and $\Upsilon$ to all of $\hat W_\alpha$. We do this recursively. But first we need to prepare the uncoupled parts of the images $\cF_{s_0}\hat W_\alpha$ and $\cF_{s_0}\hat W_\beta$ so that they also can undergo the coupling procedure described above. 

On the one hand, we have the sets $\cF_{s_0}\hat W_\alpha\setminus\cup_i \hat W_{\alpha,s_0,i}$ and  $\cF_{s_0}\hat W_\beta\setminus \cup_i \hat W_{\beta,s_0,i}$ which consist of several rectangles whose base curves do \emph{not} represent proper crossings of the magnet $\fM_{s_0}$. In the worst case, such a base curve is the excess piece of a longer curve that has crossed the magnet $\fM_{s_0}$ properly, but by definition such excess pieces have a uniform lower bound on their length.

On the other hand, we also have the sets $\hat W_{\alpha,s_0,i}\setminus \hat W_{\alpha,s_0,i}'=\{(x,t)\in \hat W_{\alpha,s_0,i} \,:\, \tau_\alpha<t\leq 1 \}$ and $\hat W_{\beta,s_0,i}\setminus\hat W_{\beta,s_0,i}'=\{(y,s)\in \hat W_{\beta,s_0,i} \,:\, \tau_{\beta,i}(y)<s\leq 1\}$ whose bottom complements were coupled already. We stretch each $\hat W_{\alpha,s_0,i}\setminus \hat W_{\alpha,s_0,i}'$ affinely onto the complete rectangle $\hat W_{\alpha,s_0,i}$, replacing the density $\rho_\alpha$ on it by $(1-\tau_\alpha)\rho_\alpha$. Similarly, we stretch each $\hat W_{\beta,s_0,i}\setminus\hat W_{\beta,s_0,i}'$ onto $\hat W_{\beta,s_0,i}$ so that each vertical fiber $\{s \,:\, (y,s)\in \hat W_{\beta,s_0,i},\,\tau_{\beta,i}(y)<s\leq 1\}$ is mapped affinely onto $[0,1]$, and replace the density $\rho_\beta$ by the density $(1-\tau_{\beta,i})\rho_\beta$. These transformations are measure preserving; the pushforward of the original measure on the incomplete rectangle is precisely the new measure on the complete rectangle. Notice that the base curves $W_{\slot,s_0,i}$ are quite short --- of the size of the magnet --- but nevertheless have a uniform lower bound on their length.

In conclusion, a finite \emph{recovery time} $r_0$ will be sufficient for the map $\cT_{s_0+r_0,s_0+1}$ to stretch the base curves of all the remaining rectangles above to standard length. In fact, some may grow too long but can then be standardized by cutting into shorter pieces, as has been discussed earlier.

We still need to address the issue of regularity of $(1-\tau_\alpha)\rho_\alpha$ and $(1-\tau_{\beta,i})\rho_\beta$ as well as show that $\tau_{\beta,i}$ is actually well defined, \ie, that its values do not exceed $1$.
\begin{lem}\label{lem:tau_reg}
Let us take $\eta_\mathrm{r}=\eta_\mathrm{\bfh}$ \footnote{This fixes the value of $\eta_\mathrm{r}$.}. Then $\sup \tau_{\beta,i}\leq 1$. Taking the recovery time $r_0$ sufficiently long, the densities $(1-\tau_\alpha)\rho_\alpha$ and $(1-\tau_{\beta,i})\rho_\beta$ become regular under $\cT_{s_0+r_0,s_0+1}$.
\end{lem}

\begin{proof}
Throughout the proof we assume that $s_0$ is sufficiently large to begin with. 

By (the proof of) Lemma~\ref{lem:standard_image}, $\rho_\alpha$ and $\rho_\beta$ are regular densities on the curves $W_{\alpha,s_0,i}$ and $W_{\beta,s_0,i}$, respectively. Since multiplication by a \emph{constant} preserves the regularity of a density, $(1-\tau_\alpha)\rho_\alpha$ is regular. 
Showing that $(1-\tau_{\beta,i}(y))\rho_\beta(y)$ is regular requires some analysis.

Recall that the holonomy map $\bfh$ maps $W_{\alpha,s_0,i}$ onto $W_{\beta,s_0,i}$ by sliding along the connecting leaves of the stable foliation $\cW^{s_0}$. Both curves as well as the leaves are inside the magnet. Assuming that the magnet is sufficiently small, $\bfh$ is as close to the identity as we wish: given any $\delta>0$, we may assume that $|\cJ\bfh - 1|\leq \delta$. This follows immediately from Lemma~\ref{lem:holonomy_Jac_bound}, because it can be applied to the $k$-step pullback of $\bfh$ that maps $\cT_{s_0,s_0-k+1}^{-1} W_{\alpha,s_0,i}$ onto $\cT_{s_0,s_0-k+1}^{-1} W_{\beta,s_0,i}$ with $k$ large and the connecting leaves of $\cW^{s_0-k}$ sufficiently short (shorter than $\ell_0$) for Lemma~\ref{lem:holonomy_Jac_bound} to apply.

For each $x\in W_{\alpha,s_0,i}$ denote $y=\bfh x\in W_{\beta,s_0,i}$. As $|W_{\beta,s_0,i} (y_1,y_2)| = \int_{ W_\alpha (x_1,x_2) }\cJ\bfh\,dm_{W_{\alpha,s_0,i}}$, 
\beq\label{eq:length_rel}
\begin{split}
 (1-\delta) |W_{\alpha,s_0,i} (x_1,x_2)| \leq |W_{\beta,s_0,i} (y_1,y_2)| 
\leq (1+\delta) |W_{\alpha,s_0,i} (x_1,x_2)|.
\end{split}
\eeq
Observe that from \eqref{eq:regular_density} follows easily 
\beqn
e^{-C_\mathrm{r} |W_{\slot,s_0,i}|^{\eta_\mathrm{r}}}\leq \frac{|W_{\slot,s_0,i}|}{\nu_{\slot,s_0,i}(W_{\slot,s_0,i})}\,\rho\slot\leq e^{C_\mathrm{r} |W_{\slot,s_0,i}|^{\eta_\mathrm{r}}}
\eeqn
or
\beqn
\rho_\slot = \frac{\nu_{\slot,s_0,i}(W_{\slot,s_0,i})}{|W_{\slot,s_0,i}|}(1+\cO(|W_{\slot,s_0,i}|^{\eta_\mathrm{r}})).
\eeqn
By making the magnet small, the values of $\rho_\slot$ on $W_{\slot,s_0,i}$ are thus as close to its average as we wish. By \eqref{eq:tau_beta},
\beqn
\tau_{\beta,i}(y)\leq \frac{\tau_\alpha}{1-\delta}\frac{\rho_\alpha(x)}{\rho_{\beta}(y)}\leq
 \frac{\tau_\alpha}{1-\delta}\frac{1+\cO(|W_{\alpha,s_0,i}|^{\eta_\mathrm{r}})}{1+\cO(|W_{\beta,s_0,i}|^{\eta_\mathrm{r}})} \frac{|W_{\beta,s_0,i}|}{|W_{\alpha,s_0,i}|} \frac{Z_{\alpha,s_0}}{Z_{\beta,s_0}} \leq \frac12 \cdot \frac{1+\delta}{1-\delta}\frac{1+\cO( |W_{\alpha,s_0,i}|^{\eta_\mathrm{r}})}{1+\cO( |W_{\beta,s_0,i}|^{\eta_\mathrm{r}})} ,
\eeqn
where we have recalled \eqref{eq:rel_mass}, \eqref{eq:tau_alpha}, $Z_{\beta,s_0}=\nu_\beta(W_{\beta,s_0,\star} ) > d_0$, and \eqref{eq:length_rel}. The right-hand side can be made arbitrarily close to $\frac12$, so that we can take, say, $\sup \tau_{\beta,i}\leq \frac34$.

From \eqref{eq:tau_beta} and Lemma~\ref{lem:holonomy_reg} it then follows that
\beqn
\begin{split}
 | \ln (\tau_{\beta,i}(y_1)\rho_\beta(y_1)) - \ln (\tau_{\beta,i}(y_2)\rho_\beta(y_2)) | 
& 
\leq  | \ln \rho_\alpha(x_1) - \ln \rho_\alpha(x_2) | + |\ln \cJ\bfh(x_2) - \ln \cJ\bfh(x_1)| 
\\
& \leq C_\mathrm{r} | W_{\alpha,s_0,i} (x_1,x_2) |^{\eta_\mathrm{r}} + C_\mathrm{\bfh} | W_{\alpha,s_0,i} (x_1,x_2) |^{\eta_\bfh}
\\
& \leq \left(C_\mathrm{r} + C_\mathrm{\bfh} \right)(1-\delta)^{-{\min({\eta_\mathrm{r}},\eta_\bfh)}}| W_{\beta,s_0,i} (y_1,y_2) |^{\min({\eta_\mathrm{r}},\eta_\bfh)}.
\end{split}
\eeqn
Hence, $\ln (\tau_{\beta,i}\rho_\beta)$ is H\"older and then so is $\ln\tau_{\beta,i} = \ln(\tau_{\beta,i}\rho_\beta)-\ln\rho_\beta$. Using the estimates
\beqn
\min(a,b) |\ln a-\ln b|\leq |a-b| \leq \max(a,b)|\ln a-\ln b|\qquad\quad a,b>0
\eeqn
obtained from the mean-value theorem,
\beqn
\begin{split}
& | \ln(1-\tau_{\beta,i}(y_1)) -  \ln(1-\tau_{\beta,i}(y_2))| \\
& \qquad \qquad\leq \frac{| \tau_{\beta,i}(y_1) -  \tau_{\beta,i}(y_2) |}{1-\sup \tau_{\beta,i}} 
\leq \frac{ \sup \tau_{\beta,i}}{1-\sup \tau_{\beta,i}} | \ln \tau_{\beta,i}(y_1) - \ln \tau_{\beta,i}(y_2) |
\\
& \qquad \qquad \leq 3(2C_\mathrm{r} + C_\mathrm{\bfh} )(1-\delta)^{-{\min({\eta_\mathrm{r}},\eta_\bfh)}}|W_{\beta,s_0,i} (y_1,y_2) |^{\min({\eta_\mathrm{r}},\eta_\bfh)} .
\end{split}
\eeqn
A similar estimate is obtained for $(1-\tau_{\beta,i})\rho_\beta$. The H\"older constant is too large for the density to be regular, but (the proof of) Lemma~\ref{lem:standard_image} guarantees that it will become regular after a finite number, $r_0$, of time steps. 
\end{proof}

Finally, at time $s_0+r_0$, we normalize the measures on all the rectangles to probability measures thereby modifying the factor measures (see Introduction) associated with the rectangle families. As a result, we have two new standard families that can be coupled just as the original ones.

\subsection{Exponential tail bound}
For the standard families $\cG$ and $\cE$, the first coupling is constructed at time $s_0$, when enough mass of each family is on the magnet $\fM_{s_0}$: 
\beqn
\hat \mu_\cG(\Upsilon = s_0) = \frac{d_0}{2}.
\eeqn
After every coupling, there is a recovery period of $r_0$ steps, during which curves too short can grow to acceptable (\ie, standard) length and densities get regularized sufficiently. After recovery, another $s_0$ iterations are required to bring enough mass from each standard family on a magnet for the next coupling to be constructed. At the moment of the $(k+1)$st coupling,
\beqn
\hat \mu_\cG(\Upsilon = k(s_0+r_0)+s_0\,|\, \Upsilon>(k-1)(s_0+r_0)+s_0) = \frac{d_0}{2}
\eeqn
is the fraction of the previously uncoupled mass of each standard family $\cG$ and $\cE$ which lies on the magnet $\fM_{k(s_0+r_0)+s_0}$ and becomes coupled.
Hence, 
\beqn
\hat \mu_\cG(\Upsilon = k(s_0+r_0)+ s_0) = \frac{d_0}{2}\left(1-\frac{d_0}{2}\right)^k.
\eeqn

This finishes the proof of the Coupling Lemma.\qed

\appendix

\section{Subspace distance}
A natural notion of distance between subspaces $A,B\subset \bR^M$ is obtained by comparing orthogonal projections to the subspaces in the operator norm:
\beq\label{eq:gap_dist}
\dist(A,B) = \norm{P_A-P_B},
\eeq
where $P_\slot$ are the corresponding orthogonal projections. Notice that
\beq\label{eq:proj_sum}
P_{A\oplus B} = P_A+P_B, \quad \text{if $A\perp B$}.
\eeq

We will also measure the distance between 1-dimensional subspaces using the metric
\beq\label{eq:dist'}
\dist'(A,B) = \min_{\substack{u\in A,\, v\in B:\,\|u\|=\|v\|=1}} \norm{u-v} = \sqrt2 \left( 1-\langle A,B \rangle \right)^{1/2},
\eeq
where
\beqn
\average{A,B}=\max_{\substack{u\in A,\, v\in B:\,\|u\|=\|v\|=1}}\average{u,v}.
\eeqn
Let $u\in A$ and $v\in B$ be unit vectors such that $\langle u,v \rangle \geq 0$. Then
\beqn
\norm{P_A v - P_B v} = 1-\langle u,v \rangle^2 \geq 1-\langle u,v \rangle = \frac{1}{2} \norm{u-v}^{2}
\eeqn
Hence,
\beqn
\dist'(A,B)\leq \sqrt2\dist(A,B)^{1/2}.
\eeqn


\section{Uniform H\"older continuity of the (un)stable distribution}\label{app:Holder}

\begin{lem}\label{lem:distr_Holder}
For all $n$, the distributions $E^n$ and $F^n$ (see Section~\ref{sec:foliations}) are H\"older continuous with the same parameters, and the latter do not depend on the choice of the sequence $(T_i)$. 
\end{lem}
Before giving the proof, we need two auxiliary lemmas.

\begin{lem}\label{lem:attracting_cone}
For $1\leq q\leq Q$, there exist constants $k_q\in\N$ and $ 0<C_q'<1$ such that the following holds when $\ve_q$ is small enough. If each $T_i\in \cU_q$ for a fixed $q$ and if $v\in T_x\cM\setminus \cC^s_{q,x}$, then $\norm{D_x \cT_n v} \geq  C_q' \Lambda_q^n \norm{v}$ for all $n\geq 1$ and $D_x \cT_n v \in \cC^u_{q,\cT_n x}$ for all $n\geq k_q$. These statements are uniform in $x$. 
\end{lem}
\begin{proof}
If $v \in T_x\cM\setminus \cC^s_{q,x}$, we have $\norm{v^u} > a_q \norm{v^s}$. But $v_n^{u,s} \equiv D_x \widetilde T_q^n v^{u,s} \in E^{u,s}_{q,\widetilde T_q^n x}$, and $\norm{v_n^u}\geq C_q\Lambda_q^n \norm{v^u}$ and $\norm{v^s} \geq C_q\Lambda_q^n \norm{v_n^s}$. We have $\norm{v_n^u}\geq C_q^2a_q \Lambda_q^{2n}  \norm{v_n^s} \geq 2 a_q^{-1}\norm{v_n^s}$ for $n\geq k_q$, if $k_q$ is sufficiently large. Because $\norm{D_x \cT_{k_q} - D_x \widetilde T_q^{k_q}}\leq C \ve_q$, $\norm{D_x \cT_{k_q} v^s} \geq  c\norm{v^s}$, and $\norm{D_x \widetilde T_q^{k_q} v^u} \geq  c\norm{v^u}$ hold with some $C=C(k_q)$ and $c=c(k_q)$, we have
\beqn
\begin{split}
\norm{D_x \cT_{k_q}v^u} &\geq \norm{D_x \widetilde T_q^{k_q}v^u}-C \ve_q\norm{v^u} \geq \norm{D_x \widetilde T_q^{k_q}v^u} (1 - Cc^{-1}\ve_q) \\
&\geq 2 a_q^{-1}\norm{D_x \widetilde T_q^{k_q}v^s} (1 - Cc^{-1}\ve_q) \\
&\geq 2 a_q^{-1}(\norm{D_x \cT_{k_q}v^s} -C\ve_q\norm{v^s}) (1 - Cc^{-1}\ve_q) \\
&\geq 2 a_q^{-1}\norm{D_x \cT_{k_q}v^s}(1 -Cc^{-1}\ve_q)^2 
\geq a_q^{-1}\norm{D_x \cT_{k_q}v^s},
\end{split}
\eeqn
provided $\ve_q$ is small enough. This estimate shows that $D_x \cT_{k_q} v\in \cC_{q,\cT_{k_q} x}^u$.

The uniform estimate $\norm{D_x \cT_n v} \geq  c_q \norm{v}$ holds with some $c_q=c_q(k_q)<1$ for $1\leq n < k_q$. If $n = k_q+m$, $\norm{D_x \cT_n v} \geq C_q\Lambda_q^m\norm{D_x \cT_{k_q}  v} \geq c_q C_q\Lambda_q^m \norm{v}$. Hence, we can set $C_q' = c_qC_q/\Lambda_q^{k_q}$, so that $\norm{D_x \cT_n v} \geq  C_q' \Lambda_q^n \norm{v}$ for all $n\geq 1$.
\end{proof}

The next result on linear maps is cited from \cite[Lemma~6.1.1]{BrinStuck} up to notational changes.
\begin{lem}\label{lem:BrinStuck}
Let $\cL_n^1$ and $\cL_n^2$, $n\in\bN$, be two sequences of linear maps $\bR^M\to\bR^M$. Assume that for some $b>0$ and $\delta\in(0,1)$,
\beqn
\norm{\cL_n^1-\cL_n^2} \leq \delta b^n,\quad n\geq 0.
\eeqn
Suppose there are two subspaces $\cE^1$ and $\cE^2$ of $\bR^M$ and constants $C_\star>1$, $0<\lambda_\star<\mu_\star$ with $\lambda_\star<b$ such that
\beqn
\begin{cases}
\norm{\cL_n^i v}\, \leq C_\star\lambda_\star^n &\text{if $v\in \cE^i$}, \\
\norm{\cL_n^i w} \geq C_\star^{-1}\mu_\star^n &\text{if $w\perp \cE^i$}.
\end{cases}
\eeqn
Then
\beqn
\dist(\cE^1,\cE^2) \leq 3C_\star^2\frac{\mu_\star}{\lambda_\star}\delta^{(\ln\mu_\star-\ln\lambda_\star)/(\ln b - \ln \lambda_\star)}.
\eeqn
\end{lem}

\begin{proof}[Proof of Lemma~\ref{lem:distr_Holder}]
We generalize the case of a single map found in \cite{BrinStuck}.
As the orthogonal complements $(T_x\cM)^\perp$ of $T_x\cM$ in $\bR^M$ form a smooth distribution $(T\cM)^\perp$ on $\cM$, we first prove that the distribution $E^0\oplus (T\cM)^\perp$ is H\"older continuous on $\cM$ and then deduce the H\"older continuity of $E^0$.

Let $P_x$ be the orthogonal projection $\bR^M\to T_x\cM$. It depends smoothly on $x$.
We define 
\beqn
L^{(i)}(x) = D_xT_i \circ P_x.
\eeqn
This map extends $D_xT_i$ to a linear map $\bR^M\to \bR^M$. Let us also set
\beqn
\cL_n(x) = L^{(n)}(\cT_{n-1}x)\cdots L^{(1)}(x).
\eeqn

Recall from above that $E_x^0 \subset \cC^{s}_{1,x}$. Setting $C = \prod_{1\leq q\leq Q}C_q$  and $\Lambda = \min_{1\leq q\leq Q}\Lambda_q$, we have
\beqn
\norm{D_x \cT_n v} \leq  C^{-1} \Lambda^{-n} \norm{v},\qquad v\in E_x^0,
\eeqn
 by (A2). This translates to
\beq\label{eq:BrinStuck_cond_1a}
\norm{\cL_n(x) v} \leq  C^{-1} \Lambda^{-n} \norm{v},\qquad v\in E_x^0\oplus (T_x\cM)^\perp.
\eeq
By Assumption (A4), we can make the cones so narrow that a vector in $T_x\cM$ perpendicular to $E_x^0$ lies in the \emph{complement} of $\cC^{s}_{1,x}$. Setting $C' = \prod_{1\leq q\leq Q}C'_q$, Lemma~\ref{lem:attracting_cone} thus yields 
\beqn
\norm{D_x \cT_n w} \geq  C'\Lambda^n \norm{w},\qquad w\in T_x\cM:w\perp E_x^0.
\eeqn
In other words,
\beq\label{eq:BrinStuck_cond_1b}
\norm{\cL_n(x) w} \geq  C'\Lambda^n \norm{w},\qquad w\perp E_x^0\oplus (T_x\cM)^\perp.
\eeq

Clearly
 $\norm{L^{(i)}(x)} = \norm{D_x T_i}$. For brevity, let us write $b_1 = \sup_{T}\sup_{x\in\cM}\norm{D_x T}$ and $b_2 = \sup_{T}\sup_{x\in\cM}\norm{D_x (D_x T\circ P_x)}$, where $T$ runs over $\cU_1\cup\dots\cup\cU_Q$.  Since $\cL_{n+1}(x) = L^{(n+1)}(\cT_n x) \cL_n(x)$,
\beqn
\begin{split}
& \norm{\cL_{n+1}(x) - \cL_{n+1}(y)} 
\\
&\qquad \leq \norm{L^{(n+1)}(\cT_n x)}\,\norm{\cL_n(x) - \cL_n(y)} + \norm{L^{(n+1)}(\cT_n x) - L^{(n+1)}(\cT_n y)} \, \norm{\cL_n(y)} 
\\ 
& \qquad \leq b_1 \norm{\cL_n(x) - \cL_n(y)} + b_2 \norm{\cT_n x - \cT_n y} \,b_1^n
\leq b_1 \norm{\cL_n(x) - \cL_n(y)} + b_2 b_1^{2n} \norm{x - y}.
\end{split}
\eeqn
As $\|\cL_1(x)-\cL_2(y)\|\leq b_2\|x-y\|$, we obtain the bound
\beq\label{eq:BrinStuck_cond_2}
\norm{\cL_{n}(x) - \cL_{n}(y)} \leq c \norm{x-y} b_1^{2n},
\eeq
with the constant $c=\frac{b_2}{b_1(b_1-1)}$. 

The bounds \eqref{eq:BrinStuck_cond_1a}, \eqref{eq:BrinStuck_cond_1b}, and \eqref{eq:BrinStuck_cond_2} show that all conditions of Lemma~\ref{lem:BrinStuck} are satisfied, if we take $\cL_n^1=\cL_n(x)$, $\cL_n^2=\cL_n(y)$, $\cE^1=E^0_x\oplus (T_x\cM)^\perp$, and $\cE^2=E^0_y\oplus (T_y\cM)^\perp$. Writing $\alpha = 2\ln\Lambda/(2\ln b_1+\ln\Lambda)$ and $K = 3\max(C^{-1},C'^{-1})^2 \Lambda^2 \,c^\alpha$,
\beqn
\begin{split}
& \dist\! \left(E^0_x\oplus (T_x\cM)^\perp,E^0_y\oplus (T_y\cM)^\perp\right) \leq K\norm{x-y}^{\alpha}
\end{split}
\eeqn
provided $\norm{x-y}<1/c$.

By compactness of the manifold and smoothness of the distribution $(T\cM)^\perp$, we have \linebreak $\dist((T_x\cM)^\perp,(T_y\cM)^\perp)\leq L\norm{x-y}$ for some $L$. Hence, by \eqref{eq:gap_dist} and \eqref{eq:proj_sum},
\beqn
\dist(E_x^0,E_y^0) \leq (K+L)\norm{x-y}^\alpha \qquad \text{if $\norm{x-y}< 1/c$}.
\eeqn

Notice that this bound does not depend on the sequence $(T_i)_{i\geq 1}$. Moreover, the same upper bound is obtained for each distribution $E^n$ by disregarding the first $n$ maps and considering the sequence $(T_i)_{i> n}$ instead. The result for $F^n$ is obtained by reversing time.
\end{proof}


\section{Inclination Lemma type results}

\begin{lem}\label{lem:exp_comp}
Fix $1\leq q\leq Q$ and let $T_i\in \cU_q$ for each $i$. There exists a constant $C_\#>0$ such that, for all $w\in\cC_{q,x}^u$ and all $\tilde w\in T_x\cM$ with $\|w\|=\|\tilde w\|=1$, the vectors $w_n=D_x\cT_n w$ and $\tilde w_n = D_x\cT_n \tilde w$ satisfy
\beqn
\frac{\|\tilde w_n\|}{\|w_n\|} \leq C_\#\qquad \forall\, n\geq 0.
\eeqn
If also $\tilde w\in\cC_{q,x}^u$, the angle between $w_n$ and $\tilde w_n$ tends to zero at a uniform exponential rate: 
\beq\label{eq:angle_conv}
 1- \left| \left\langle \frac{w_n}{\|w_n\|},\frac{\tilde w_n}{\|\tilde w_n\|} \right\rangle \right| \leq \min(c,C\Lambda_q^{-4n})
\eeq
for some $0<c<1$ and $C>0$.
\end{lem}
\begin{proof}
We first construct ``fake'' stable and unstable distributions for finite sequences $T_1,\dots,T_N$. Fix $N\geq 1$.
First, choose a distribution $\tilde E^N$ such that $\tilde E^N_x\subset \cC^s_{q,x}$ for all $x$. Then define the distributions $\tilde E^n$, $0\leq n<N$, by pulling back: $\tilde E_x^n = D_{T_{n+1} x} T_{n+1}^{-1}\tilde E_{T_{n+1} x}^{n+1} \subset \cC^s_{q,x}$. Next, choose a distribution $\tilde F^0$ such that $\tilde F^0_x\subset \cC^u_{q,x}$ for all $x$. Then set recursively $\tilde F^n_x = D_{T_n^{-1}x}T_n\, \tilde F^{n-1}_{T_n^{-1}x}$ for $n\geq 1$. We will use $\tilde E^n_x$ and $\tilde F^n_x$ as coordinate axes. Let $\fe^n_x\in \tilde E^n_x$ and $\ff^n_x\in \tilde F^n_x$ be unit vectors oriented so that $D_xT_n \fe^n_x$ and $\fe^{n+1}_{T_nx}$  point in the same direction and $D_xT_n \ff^n_x$ and $\ff^{n+1}_{T_nx}$ point in the same direction. 

Recall from Section~\ref{subsec:Anosov_comp} that the angle between $E^u_{q,x}$ and $E^s_{q,x}$ is uniformly bounded away from zero. In other words, there exists a $\psi_q>0$ such that $\average{E^u_{q,x},E^s_{q,x}}\leq 1-2\psi_q$ for all $x\in \cM$. By Assumption (A4), the cones can be assumed narrow enough, so that
$
\average{U,V}\leq 1-\psi_q
$
for all subspaces $U\subset \cC^u_{q,x}$ and $V\subset \cC^s_{q,x}$, and for all $x\in\cM$. 
Because of 
this, $|\langle \fe^n_x, \ff^n_x \rangle|\leq 1-\psi_q$ for all $x$ and all $0\leq n\leq N$. Thus, we have the uniform bounds
\beq\label{eq:length_comp}
\psi_q(\alpha^2+\beta^2) \leq \|\alpha \ff^n_x+\beta \fe^n_x\|^2 \leq 2(\alpha^2+\beta^2), \qquad\forall\,\alpha,\beta, \quad 0\leq n\leq N.
\eeq
Now write $ w_n = \alpha_n \ff^n_{\cT_n x}+\beta_n \fe^n_{\cT_n x}$ and $ \tilde w_n = \tilde \alpha_n \ff^n_{\cT_n x}+ \tilde \beta_n \fe^n_{\cT_n x}$ and estimate, for $0\leq n\leq N$,
\beqn
\frac{\| \tilde w_n\|^2}{\| w_n\|^2}   \leq \frac{2(\tilde \alpha_n^2 + \tilde \beta_n^2)}{\psi_q^2\alpha_n^2} = \frac{2\tilde\alpha_0^2}{\psi_q^2\alpha_0^2} + \frac{2\tilde\beta_0^2\|D_x\cT_n \fe^0_x\|^2}{\psi_q^2\alpha_0^2\|D_x\cT_n \ff^0_x\|^2} \leq \frac{2\tilde\alpha_0^2}{\psi_q^2\alpha_0^2} + \frac{2\tilde\beta_0^2(C_q\Lambda_q^n)^{-2}}{\psi_q^2\alpha_0^2(C_q\Lambda_q^n)^2}.
\eeqn
We have used Assumption (A2) to bound the norms involving $D_x\cT_n$. From \eqref{eq:length_comp}, $\tilde \alpha_0^2,\tilde\beta_0^2\leq \psi_q^{-1}$. As $w$ is an unstable unit vector, $|\alpha_0|$ is bounded from below by some $A_q>0$.
Thus
\beqn
\frac{\|\tilde w_n\|^2}{\| w_n\|^2} \leq \frac{2}{\psi_q^3A_q^2}(1+C_q^{-4}) \equiv C_\#^2 \qquad 0\leq n\leq N.
\eeqn
But $N$ was arbitrary, so the bound holds for all $n\geq 0$. In particular, $C_\#$ does not depend on the constructed distributions. A computation also shows that, if both $\tilde w,w\in\cC_{q,x}^u$, then $1-|\langle w_n,\tilde w_n\rangle|/\|w_n\| \|\tilde w_n\|$ is of order $\beta_n\tilde \beta_n/\alpha_n\tilde\alpha_n \leq C_q^{-4}\Lambda_q^{-4n}\beta_0\tilde \beta_0/\alpha_0\tilde \alpha_0\leq C\Lambda_q^{-4n}$. 
\end{proof}

\begin{lem}\label{lem:coalesce}
Fix $1\leq q\leq Q$, $0<\lambda<1$, and $\delta>0$. Fix $N\geq 1$ and $T_i\in \cU_q$ for $1\leq i\leq N$. Take two points $x_1,x_2$ and assume that $d(\cT_n x_1,\cT_n x_2)<\delta\lambda^n$ if $0\leq n\leq N$. Suppose $W_1,W_2$ are unstable curves with respect to $\{\cC^u_{q,x}\}$ and that $x_i\in W_i$. Then
\beq\label{eq:coalesce}
\left| \frac{\cJ_{\cT_nW_1} T_{n+1} (\cT_n x_1)}{\cJ_{\cT_n W_2}T_{n+1}(\cT_n x_2) } -1  \right| \leq C' \mu^n \qquad 0\leq n\leq N-1.
\eeq
The constants $C'>1$ and $0<\mu<1$ are independent of $N$, of the curves $W_i$, and of the choice of $T_1,\dots,T_N$, as long as the bound on $d(\cT_n x_1,\cT_n x_2)$ continues to hold.
\end{lem}
\begin{proof}
Choose $\tilde F^0_x = E^u_{q,x}$ and define $\tilde F^n_x = D_{T_n^{-1}x}T_n\, \tilde F^{n-1}_{T_n^{-1}x}$ for $1\leq n\leq N$. By Lemma~\ref{lem:distr_Holder}~\footnote{Lemma~\ref{lem:distr_Holder} has been formulated for $F^n$ as defined in \eqref{eq:fake_unstable}. Considering the special case $Q=1$, we can clearly recover the claimed result for $\tilde F^n$ as defined here.}, the distributions $\tilde F^n$ belong to a fixed H\"older class, no matter which $T_1,\dots,T_N$ and $N$ are chosen. Notice that $\tilde F^n_x \subset \cC^u_{q,x}$ for $0\leq n\leq N$. 

Let $u_i^n$ ($i=1,2$) stand for a unit tangent vector of $\cT_n W_i$ at $x_i^n=\cT_n x_i$. We also write $L(x)=D_xT_{n+1}\circ P_x$, where $P_x$ is the orthogonal projection $\bR^M\to T_x \cM$, which is smooth. \beqn
\begin{split}
\left| \frac{\cJ_{\cT_nW_1} T_{n+1} (x_1^n)}{\cJ_{\cT_n W_2}T_{n+1}(x_2^n) } -1  \right| &\leq \frac{1}{C_q}\left| \norm{D_{x_1^n}T_{n+1}u_1^n }-\norm{D_{x_2^n}T_{n+1}u_2^n }\right|= \frac{1}{C_q}\left| \norm{L(x_1^n)u_1^n }-\norm{L(x_2^n)u_2^n }\right|
\\
& \leq \frac{1}{C_q} \left( \norm{L(x_1^n)-L(x_2^n)} + \norm{L(x_2^n)} \min_{\sigma=\pm1} \norm{u_1^n-\sigma u_2^n} \right)
\\
& \leq C\left(  \lambda^n + \dist'(U_1^n,U_2^n) \right).
\end{split}
\eeqn
We have denoted by $U_i^n$ the linear subspaces of $\bR^M$ spanned by $u_i^n$ and recalled the definition in \eqref{eq:dist'}. The angle between $U_i^n$ and $\tilde F^n_{x_i^n}$ decays exponentially: $\dist'(U_i^n,\tilde F^n_{x_i^n})\leq C \Lambda_q^{-2n}$ due to \eqref{eq:dist'} and \eqref{eq:angle_conv}. Hence, it suffices to prove exponential decay of $\dist'(\tilde F^n_{x_1^n},\tilde F^n_{x_2^n})$. But this follows from H\"older continuity and the assumption $d(x_1^n,x_2^n)<\delta\lambda^n$.
\end{proof}


\bibliography{Anosov_coupling}{}

\begin{thebibliography}{10}

\bibitem{ArnouxFisher}
Pierre Arnoux and Albert~M. Fisher.
\newblock Anosov families, renormalization and non-stationary subshifts.
\newblock {\em Ergodic Theory Dynam. Systems}, 25(3):661--709, 2005.

\bibitem{AyyerStenlund}
Arvind Ayyer and Mikko Stenlund.
\newblock Exponential decay of correlations for randomly chosen hyperbolic
  toral automorphisms.
\newblock {\em Chaos}, 17(4):043116, 7, 2007.

\bibitem{Bakhtin1}
Victor~I. Bakhtin.
\newblock Random processes generated by a hyperbolic sequence of mappings. {I}.
\newblock {\em Izv. Ross. Akad. Nauk Ser. Mat.}, 58(2):40--72, 1994.

\bibitem{Bakhtin2}
Victor~I. Bakhtin.
\newblock Random processes generated by a hyperbolic sequence of mappings.
  {II}.
\newblock {\em Izv. Ross. Akad. Nauk Ser. Mat.}, 58(3):184--195, 1994.

\bibitem{BressaudLiverani}
Xavier Bressaud and Carlangelo Liverani.
\newblock Anosov diffeomorphisms and coupling.
\newblock {\em Ergodic Theory Dynam. Systems}, 22(1):129--152, 2002.

\bibitem{BrinStuck}
Michael Brin and Garrett Stuck.
\newblock {\em Introduction to dynamical systems}.
\newblock Cambridge University Press, Cambridge, 2002.

\bibitem{Chernov-BilliardsCoupling}
Nikolai Chernov.
\newblock Advanced statistical properties of dispersing billiards.
\newblock {\em J. Stat. Phys.}, 122(6):1061--1094, 2006.

\bibitem{ChernovDolgopyatBBM}
Nikolai Chernov and Dmitry Dolgopyat.
\newblock Brownian {B}rownian motion. {I}.
\newblock {\em Mem. Amer. Math. Soc.}, 198(927):viii+193, 2009.

\bibitem{ConzeRaugi}
Jean-Pierre Conze and Albert Raugi.
\newblock Limit theorems for sequential expanding dynamical systems on
  {$[0,1]$}.
\newblock In {\em Ergodic theory and related fields}, volume 430 of {\em
  Contemp. Math.}, pages 89--121. Amer. Math. Soc., Providence, RI, 2007.

\bibitem{KatokHasselblatt}
Anatole Katok and Boris Hasselblatt.
\newblock {\em Introduction to the modern theory of dynamical systems},
  volume~54 of {\em Encyclopedia of Mathematics and its Applications}.
\newblock Cambridge University Press, Cambridge, 1995.
\newblock With a supplementary chapter by Katok and Leonardo Mendoza.

\bibitem{MasmoudiYoung}
Nader Masmoudi and Lai-Sang Young.
\newblock Ergodic theory of infinite dimensional systems with applications to
  dissipative parabolic {PDE}s.
\newblock {\em Comm. Math. Phys.}, 227(3):461--481, 2002.

\bibitem{OttStenlundYoung}
William Ott, Mikko Stenlund, and Lai-Sang Young.
\newblock Memory loss for time-dependent dynamical systems.
\newblock {\em Math. Res. Lett.}, 16(3):463--475, 2009.

\bibitem{Young}
Lai-Sang Young.
\newblock Recurrence times and rates of mixing.
\newblock {\em Israel J. Math.}, 110:153--188, 1999.

\end{thebibliography}
\bibliographystyle{plain}

\end{document}